\documentclass[11pt]{amsart}
\usepackage{amssymb, graphicx, amsthm}
\usepackage{epsfig}
\usepackage{verbatim}
\usepackage{amsfonts}
\usepackage{mathrsfs}
\usepackage{amsbsy}
\renewcommand{\eqref}[1]{(\ref{#1})}

\newtheorem{prop}{Proposition}[section]
\newtheorem{lem}[prop]{Lemma}

\newtheorem{thm}{Theorem}[section]

\usepackage{amscd}

\begin{document}

\title[Invariant geodesics under pseudo-Anosov mapping classes]{Invariant geodesics in the curve complex under point-pushing pseudo-Anosov \\ mapping classes}
\author[C. Zhang]{C. Zhang}
\date{March 16, 2013}
\thanks{ }

\address{Department of Mathematics \\ Morehouse College
\\ Atlanta, GA 30314, USA.}
\email{czhang@morehouse.edu}


\maketitle 

\begin{abstract}
Let $S$ be a closed Riemann surface of genus $p>1$ with one point removed. In this paper, we identify those point-pushing pseudo-Anosov maps on $S$ that preserve at least one bi-infinite geodesic in the curve complex. 
\end{abstract}

\bigskip
\bigskip

\section{Introduction and statement of results}
\setcounter{equation}{0}

Let $S$ be a closed Riemann surface of genus $p>1$ with $n$ punctures removed. Assume that $3p-4+n>0$. Let Mod$(S)$ denote the mapping class group which consists of isotopy classes of orientation preserving self-homeomorphisms of $S$.  In view of the Nielsen--Thurston classification theorem \cite{Th}, elements of Mod$(S)$ are represented by periodic, reducible, or pseudo-Anosov maps. See Fathi--Laudenbach--Poenaru \cite{FLP} for the definitions and more information on reducible and pseudo-Anosov maps.  

The mapping class group Mod$(S)$ can naturally act on the Teichm\"{u}ller space $T(S)$ as a group of isometries with respect to the Teichm\"{u}ller metric $d_T$. Royden's theorem \cite{Ro}, whose generalization is due to Earle--Kra \cite{E-K}, asserts that with a few exceptions, the group of automorphisms of $T(S)$ is the group Mod$(S)$. Following Bers \cite{Bers2} elements $\alpha\in \mbox{Mod}(S)$ can be classified as elliptic, parabolic, hyperbolic, or pseudo-hyperbolic elements with the aid of the index $a(\alpha)=\inf \{d_T(y,\alpha(y)): y\in T(S)\}$. That is, $\alpha$ is elliptic if there is $y_0\in T(S)$ such that  $a(\alpha)=d_T(y_0,\alpha(y_0))=0$; parabolic if $a(\alpha)=0$ but $d_T(y,\alpha(y))>0$ for all $y\in T(S)$; hyperbolic if there is $y_0\in T(S)$ such that  $a(\alpha)=d_T(y_0,\alpha(y_0))>0$; and pseudo-hyperbolic if $a(\alpha)>0$ and for all $y\in T(S)$, $a(\alpha)<d_T(y,\alpha(y))$. 

Bers \cite{Bers2} proved that an element $\alpha\in \mbox{Mod}(S)$ is elliptic if and only if it is represented by a periodic map; $\alpha$ is parabolic or pseudo-hyperbolic if and only if it is represented by a reducible map; and $\alpha$ is hyperbolic if and only if it is represented by a pseudo-Anosov map. Among other things, it is well known that any hyperbolic element $\alpha$ preserves a unique bi-infinite geodesic $l$ in $T(S)$ (called Teichm\"{u}ller geodesics in the literature), and hyperbolic elements are the only elements that keep some bi-infinite geodesics invariant. We remark here that the existence of $l$ was proved by Bers \cite{Bers2}; and the uniqueness of $l$ was proved in Bestvina--Feighn \cite{B-F} using topological methods.  

The mapping class group Mod$(S)$ acts on the {\it complex of curves} $\mathcal{C}(S)$ of $S$ as well, where $\mathcal{C}(S)$ is the simplicial complex whose vertex set $\mathcal{C}_0(S)$ is the collection of simple closed geodesics on $S$ and whose $k$-dimensional simplicies $\mathcal{C}_k(S)$ are the collections of $(k+1)$-tuples $(v_0,v_1,\cdots, v_k)$ of disjoint simple closed geodesics on $S$ (see Harvey \cite{H}). It is well-known that $\mathcal{C}(S)$ is connected and locally infinite. 
For simplicity, any path $\{(u, u_1), (u_1,u_2), \cdots, (u_s,v)\}$ joining two vertices $u,v\in \mathcal{C}_0(S)$ is denoted by $[u,u_1,\cdots, u_s,v]$. It is natural to define a path distance $d_{\mathcal{C}}(u,v)$ for any $u,v\in \mathcal{C}_0(S)$ to be the minimum number of sides in $\mathcal{C}_1(S)$ joining $u$ and $v$, where one of the paths that achieves the minimum length is called a geodesic segment joining $u$ and $v$. Masur--Minsky \cite{M-M} showed that $\mathcal{C}(S)$ has an infinite diameter and is $\delta$-hyperbolic in the sense of Gromov \cite{Gro}. 

When considering actions of elements of Mod$(S)$ on $\mathcal{C}(S)$, things are similar but different than that on $T(S)$. Ivanov \cite{Iva} showed that with a few exceptions, the group of automorphisms of $\mathcal{C}(S)$ is the full group Mod$(S)$. It was shown in \cite{M-M} that elements of Mod$(S)$ can be classified as elliptic and hyperbolic elements (see also \cite{Gro} for the definition and terminology). In particular, Mod$(S)$ contains no parabolic elements and hyperbolic elements are represented by pseudo-Anosov maps. 

In \cite{Bow}, Bowditch proved that there exists an integer $m$, whose precise value is unknown, such that for any hyperbolic mapping class $f$, the power $f^m$ preserves finitely many bi-infinite geodesics in $\mathcal{C}(S)$, where an infinite path $[\cdots, u_{-m}, \cdots, u_0,\cdots, u_m, \cdots]$ is called a bi-infinite geodesic if $u_{-m}$ and $u_m$ both tend to points in the Gromov boundary $\partial \mathcal{C}(S)$ of $\mathcal{C}(S)$ and for any $m$, the subpath $[u_{-m}, \cdots, u_0, \cdots, u_m]$ is a geodesic segment connecting $u_{-m}$ and $u_m$. It is quite obvious that a non periodic or a non pseudo-Anosov map does not preserve any bi-infinite geodesic. See Section 2 for more expositions. 

The question arises as to whether there exist some primitive pseudo-Anosov maps that preserve bi-infinite geodesics. 

Let $x$ be a puncture of $S$. Let $\mathscr{F}^*\subset \mbox{Mod}(S)$ be the subgroup consisting of mapping classes projecting to the trivial mapping class on $\tilde{S}=S\cup \{x\}$.  Let $\mathscr{F}\subset \mathscr{F}^*$ be the subset consisting of primitive pseudo-Anosov elements isotopic to the identity on $\tilde{S}$. Then $\mathscr{F}\neq \emptyset$ and contains infinitely many elements (Kra \cite{Kr}). 
More precisely, each primitive and oriented filling closed geodesic $\tilde{c}$ on $\tilde{S}$ (that is, $\tilde{c}$ is not a power of any other closed geodesic and intersects every simple closed geodesic on $\tilde{S}$) is associated with a conjugacy class $H(\tilde{c})$ that consists of mapping classes conjugate in $\mbox{Mod}(S)$ to the point-pushing pseudo-Anosov mapping class along the geodesic $\tilde{c}$, and  $\mathscr{F}$ is partitioned into a disjoint union of conjugacy 
classes $H(\tilde{c})$ for all primitive and oriented filling closed geodesics on $\tilde{S}$. 

Let $\mathscr{S}$ denote the set of primitive, oriented filling closed geodesics on $\tilde{S}$, and let $\mathscr{S}(2)$ be the subset of $\mathscr{S}$ consisting of filling closed geodesics that intersect every simple closed geodesic at least twice. It is easy to see that  both  $\mathscr{S}(2)$ and $\mathscr{S}\backslash \mathscr{S}(2)$ are not empty. For every $\tilde{c}\in \mathscr{S}\backslash \mathscr{S}(2)$, we denote by $\mathscr{S}_{\tilde{c}}$ the (finite) set of simple closed geodesics intersecting $\tilde{c}$ only once.  

Our aim in this paper is to investigate the actions of elements of $\mathscr{F}^*$ on $\mathcal{C}_0(S)$ and to uncover elements in $\mathscr{F}^*$ that preserve some bi-infinite geodesics in $\mathcal{C}(S)$.  In contrast to Theorem 1.3 of \cite{Bow}, we will prove the following result. 

\begin{thm}\label{T1}
Let $S$ be of type $(p,1)$ with $p>1$. We have: 

\noindent $(1)$ Elements of $\mathscr{F}^*\backslash \mathscr{F}$ do not preserve any bi-infinite geodesics in $\mathcal{C}(S)$. 

\noindent  $(2)$ Let $f\in \mathscr{F}$ be such that the corresponding filling geodesic $\tilde{c}\in \mathscr{S}\backslash \mathscr{S}(2)$. Then $f$ preserves at least one bi-infinite geodesic in $\mathcal{C}(S)$. 

\noindent $(3)$ There is a injective map $\mathscr{I}$ of $\mathscr{S}_{\tilde{c}}$  into the set of $f$-invariant bi-infinite geodesics in $\mathcal{C}(S)$ so that $\mathscr{I}(\mathscr{S}_{\tilde{c}})$ consists of disjoint bi-infinite geodesics. 
\end{thm}

\noindent {\em Remark. } It is not known whether $\mathscr{I}$ is a bijection; and whether $f\in \mathscr{F}$ preserves a bi-infinite geodesic when the corresponding filling geodesic $\tilde{c}$ is in $\mathscr{S}(2)$. 

\smallskip

The curve complex $\mathcal{C}(\tilde{S})$ along with the vertex set $\mathcal{C}_0(\tilde{S})$ and the path metric $d_{\mathcal{C}}$ on $\mathcal{C}(\tilde{S})$
can similarly be defined.  For each $\tilde{u}\in \mathcal{C}_0(\tilde{S})$, let $F_{\tilde{u}}$ denote the set of vertices $u$ in $\mathcal{C}_0(S)$ such that $u$ is freely homotopic to $\tilde{u}$ as the puncture $x$ is filled in. Let $\mathbf{H}$ be a hyperbolic plane and let $\varrho:\mathbf{H}\rightarrow \tilde{S}$ be the universal covering map with covering group $G$. Then with the help of the covering map $\varrho$, every $u\in F_{\tilde{u}}$ is associated with a configuration $(\tau_u,\Omega_u, \mathscr{U}_u)$, and every element $f\in H(\tilde{c})$ corresponds to an essential hyperbolic element $g$ of $G$. Let $\mbox{axis}(g)$ be the axis of $g$ that is the unique $g$-invariant geodesic in $\mathbf{H}$. See Section 2 for more details. 

Let $f^m(u)$ denote the geodesic freely homotopic to the image curve of $u$ under $f^m$. Theorem \ref{T1} follows from the following result. 

\begin{thm}\label{T2}
Let $\tilde{u}\in \mathcal{C}_0(\tilde{S})$ and $\tilde{c}\in \mathscr{S}$. Let $u\in F_{\tilde{u}}$ and $f\in H(\tilde{c})$ be such that 
$\Omega_u\cap \rm{axis}(g)\neq \emptyset$. Then $\tilde{u}$ intersects $\tilde{c}$ only once if and only if $d_{\mathcal{C}}(u,f^m(u))=m$ for all $m$, in which case, there is a unique geodesic segment connecting $u$ and $f^m(u)$. 
\end{thm}

This paper is organized as follows. In Section 2, we collect some basic facts about mapping class groups acting on the curve complex, as well as some background information on Bers isomorphisms. In Section 3, we refine the argument in \cite{CZ11} to estimate the lower bound for the distance $d_{\mathcal{C}}(u,f^m(u))$ in terms of the intersection number between the corresponding geodesics.  In Section 4, we relate a geodesic segment joining $u$ and $f^m(u)$ to a sequence of adjacent convex regions in $\mathbf{H}$. In Section 5, we prove the main results.

\section{Mapping class group acting on the curve complex}
\setcounter{equation}{0}

\noindent {\bf \S 2.1. } In \cite{M-M}, Masur--Minsky proved that there is a constant $\epsilon$, depending only on the type $(p,n)$ (with $3p+n-4>0$) of the surface $S$, such that for any pseudo-Anosov map $f$, any vertex $u\in \mathcal{C}_0(S)$ and any integer $m>0$, $d_{\mathcal{C}}(u,f^m(u))\geq \epsilon |m|$. From this fact together with the Nielsen--Thurston classification for mapping classes \cite{Th}, the following result is easily deduced:

\begin{lem} \cite{M-M} \label{L2.1}
Let $S$ be as above, and let $f\in \rm{Mod}(S)$. Then either $f^q$ for some $q$ has fixed points in $\mathcal{C}_0(S)$, or $f$ acts on $\mathcal{C}(S)$ as a hyperbolic translation which has two fixed points on $\partial \mathcal{C}(S)$.
\end{lem}

For the notion of hyperbolic translations, we refer to Gromov \cite{Gro}. Note that the two classes in Lemma \ref{L2.1} are exclusive. As an easy  corollary of Lemma \ref{L2.1}, we obtain

\begin{lem}\label{L2.2}
If $f\in \rm{Mod}(S)$ is reducible, then $f$ does not preserve any bi-infinite geodesic in $\mathcal{C}(S)$.
\end{lem}

\begin{proof}

Suppose that a reducible mapping class $f$ in $\mbox{Mod}(S)$ keeps an infinite geodesic $L=[\cdots, u_{-m}, \cdots, u_0, \cdots, u_m, \cdots]$ invariant, i.e., $f(L)=L$. Since $f^q$ (for all $q$) keeps only finitely many vertices in $\mathcal{C}_0(S)$, there is an integer $m>0$ such that $f^q$ for any $q$ has no fixed points on the union of the rays $[u_{m}, \cdots]\cup [\cdots, u_{-m}]$. 

By selecting a subsequence if needed, we may assume without loss of generality that (i) $f^j([u_m,\cdots])\subset [u_m,\cdots]$, (ii) $f^{-j}([\cdots,u_{-m}])\subset [\cdots,u_{-m}]$, and
(iii) as $j\rightarrow +\infty$, both $d_{\mathcal{C}}(u_{m}, f^j(u_{m}))$ and
$d_{\mathcal{C}}(u_{-m}, f^{-j}(u_{-m}))$ tend to infinity.  Note that $f^{j}(u_{m})$ and $f^{-j}(u_{-m})$ belong to $L$ and that $d_{\mathcal{C}}(u_{-m}, u_{m})$ is finite. It follows that  
$$
d_{\mathcal{C}}(f^{j}(u_{m}), f^{-j}(u_{-m}))\rightarrow +\infty
$$ 
as $m\rightarrow +\infty$.  

By taking a suitable power if necessary, we may also assume that $f(u)=u$ for some $u\in \mathcal{C}_0(S)$. Denote by $K=\mbox{max} \{ i(u,u_{m}), i(u,u_{-m})\}$. Then since $f^j$ is a homeomorphism, $K\geq i(u,u_m)=i(f^j(u),f^j(u_{m}))=i(u,f^j(u_{m}))$. Similarly, we have $K\geq i(u,f^{-j}(u_{-m}))$. From Lemma 2.1 of \cite{M-M} (or \cite{Bow1}), we conclude that $d_{\mathcal{C}}(u,f^j(u_{m}))\leq K+1$ and $d_{\mathcal{C}}(u,f^{-j}(u_{-m}))\leq K+1$. It follows from the triangle inequality that 
$d_{\mathcal{C}}(f^{-j}(u_{m}),f^j(u_{m}))< +\infty$ for all $m$. This contradicts that $d_{\mathcal{C}}(f^{j}(u_{m}), f^{-j}(u_{-m}))\rightarrow +\infty$. 
\end{proof}

\noindent {\bf \S 2.2. } Let $Q(G)$ denote the group of quasiconformal automorphisms $w$ on the hyperbolic plane $\mathbf{H}$ that satisfy $wGw^{-1}=G$. Following Bers \cite{Bers1}, two such maps $w,w'\in Q(G)$ are said equivalent if $w|_{\mathbf{S}^1}=w'|_{\mathbf{S}^1}$, where $\mathbf{S}^1$ denotes the unit circle which can be identified with the boundary of $\mathbf{H}$. Denote by $[w]$ the equivalence class of $w\in Q(G)$ and by $Q(G)/\!\sim$ the quotient group of $Q(G)$ by the above equivalence relation. The Bers isomorphism theorem (Theorem 9 of \cite{Bers1}) asserts that there is an isomorphism $\varphi^*$ of $Q(G)/\!\sim$ onto 
Mod$(S)$. For simplicity, let $[w]^*$ denote the mapping class $\varphi^*([w])$. 

It is clear that $G$ can be regarded as a normal subgroup of $Q(G)/\!\sim$.  Every hyperbolic element $h\in G$ keeps a unique geodesic in $\mathbf{H}$ invariant. This geodesic is called the axis of $h$ and is denoted by axis$(h)$. A hyperbolic element $g\in G$ is called essential if $\varrho(\mbox{axis}(g))$ is a filling closed geodesic. Let $G'\subset G$ be the collection of all primitive essential hyperbolic elements. Then $\varphi^*(G')=\mathscr{F}$ and $\varphi^*(G)=\mathscr{F}^*$. For an element $h\in G$, we denote by $h^*$ the mapping class $\varphi^*(h)$. 

Let $\pi_1(\tilde{S},x)$ denote the fundamental group of $\tilde{S}$. Let 
$\mu:G\rightarrow \pi_1(\tilde{S},x)$ denote an isomorphism (which depends only on the choice of a point $\hat{x}\in \mathbf{H}$ with $\varrho(\hat{x})=x$). By virtue of Theorem 4.1 and Theorem 4.2 in Birman \cite{Bir}, there is an exact sequence 
\begin{equation}\label{BIR}
0 \longrightarrow \pi_1(\tilde{S},x)\cong G\longrightarrow \mbox{Mod}(S)\longrightarrow \mbox{Mod}(\tilde{S})\longrightarrow 0,
\end{equation}
where $\mbox{Mod}(S)\rightarrow \mbox{Mod}(\tilde{S})$ is the natural puncture-forgetting projection. In (\ref{BIR}) an element $g\in G$ is identified with the pure mapping class in $\mbox{Mod}(S)$ that corresponds to the loop representing $\mu(g)$ in $\pi_1(\tilde{S},x)$. 

Let $u\in \mathcal{C}_0(S)$ be a non preperipheral vertex; that is, $u$ is homotopic to a non-trivial geodesic on $\tilde{S}$ if $u$ is also viewed as a curve on $\tilde{S}$. Let $\tilde{u}\in \mathcal{C}_0(\tilde{S})$ be the corresponding vertex. Denote by $\mathscr{R}_{\tilde{u}}$ the collection of all components of $\mathbf{H}\backslash \{\varrho^{-1}(\tilde{u})\}$, where 
$$\{\varrho^{-1}(\tilde{u})\}=\{\mbox{all geodesics $\hat{u}$ in }\mathbf{H}\ \mbox{such that }\varrho(\hat{u})=\tilde{u}\}.
$$
Two components $\Omega_1,\Omega_2\in \mathscr{R}_{\tilde{u}}$ are said adjacent if $\Omega_1$ and $\Omega_2$ share a common geodesic boundary $a$, that is,  $\overline{\Omega}_1\cap \overline{\Omega}_2=a$. Note that $a\in \{\varrho^{-1}(\tilde{u})\}$. It was shown (Lemma 2.1 of \cite{CZ11}) that there is a bijection 
$\chi$ between  $\mathscr{R}_{\tilde{u}}$ and $F_{\tilde{u}}$, and two regions  $\Omega_1$ and $\Omega_2\in \mathscr{R}_{\tilde{u}}$  are adjacent if and only if $d_{\mathcal{C}}(\chi(\Omega_1),\chi(\Omega_2))=1$, in which case, $\{\chi(\Omega_1), \chi(\Omega_2)\}$ forms the boundary of an $x$-punctured cylinder on $S$. That is to say, $\chi(\Omega_1)$ and $\chi(\Omega_2)$ are disjoint and homotopic to each other on $\tilde{S}$ when $\chi(\Omega_1)$ and $\chi(\Omega_2)$ are viewed as curves on $\tilde{S}$. It was shown in \cite{CZ11} that any fiber $F_{\tilde{u}}$, $\tilde{u}\in \mathcal{C}_0(\tilde{S})$, is path connected in $F_{\tilde{u}}$ (The fact that $F_{\tilde{u}}$ is connected for closed surface $\tilde{S}$ was proved in \cite{Sch}). 

Now each $u\in F_{\tilde{u}}$ is non preperipheral, which allows us to define a {\it configuration} $(\tau_u,\Omega_u, \mathscr{U}_u)$ corresponding to $u$, where $\Omega_u=\chi^{-1}(u)$, $\tau_u$ is the lift of the Dehn twist $t_{\tilde{u}}$ so that $\tau_u|_{\Omega_u}=\mbox{id}$ and $[\tau_u]^*=t_u$, and $\mathscr{U}_u$ is a partially ordered set which is the collection of all half-planes in $\mathbf{H}$ defined by $\tau_u$. Maximal elements of $\mathscr{U}_u$ are mutually disjoint, and their union is the complement of $\Omega_u$ in $\mathbf{H}$. From the construction, we also know that $\tau_u$ keeps each maximal element of $\mathscr{U}_u$ invariant. See \cite{CZ2} for more details.

\section{Distances and intersection numbers between vertices}
\setcounter{equation}{0}

Throughout the rest of the article we assume that $S$ is of type $(p,1)$ with $p>1$. This assumption guarantees that each vertex in $\mathcal{C}_0(S)$ is non preperipheral. Fix $\tilde{u}_0\in \mathcal{C}_0(\tilde{S})$ and $\tilde{c}\in \mathscr{S}$. For simplicity, we also use the symbol $i(\tilde{c}, \tilde{u}_0)$ to denote the geometric intersection number between   $\tilde{u}_0$ and $\tilde{c}$. We may assume that $\tilde{u}_0$ intersects $\tilde{c}$ at non self-intersection points of $\tilde{c}$ by performing a small perturbation if needed. Let $u_0\in \mathcal{C}_0(S)$ be obtained from $\tilde{u}_0$ by removing the point $x$. Let $(\tau_0,\Omega_0,\mathscr{U}_0)$ be the configuration that corresponds to $u_0$.
Let $g\in G$ be essential hyperbolic such that $\varrho(\mbox{axis}(g))=\tilde{c}$ and $\mbox{axis}(g)\cap \Omega_{0}\neq \emptyset$. Write $f=g^*$. Then $f\in \mathscr{F}$ is an element of $H(\tilde{c})$. 
 
By abuse of language, in what follows, for each $u\in \mathcal{C}_0(S)$,  we let $\tilde{u}$ be the corresponding vertex in $\mathcal{C}_0(\tilde{S})$ under the natural projection from $\mathcal{C}_0(S)$ onto $\mathcal{C}_0(\tilde{S})$ (which is well defined since $S$ contains only one puncture $x$), which means that $u$ and $\tilde{u}$ are homotopic to each other on $\tilde{S}$ as $x$ is filled in.  
 
The following lemma is a refinement of Theorem 1.2 of \cite{CZ10}.  

\begin{lem}\label{D}
Suppose $i(\tilde{c},\tilde{u}_0)\geq 2$.  Then for any integer $m>0$, we have $d_{\mathcal{C}}(u_0,f^m(u_0))\geq m+1$. 
\end{lem}

\begin{proof}
Since $g$ is an essential hyperbolic element of $G$, by Lemma 3.1 of \cite{CZ6}, $\mbox{axis}(g)$ is not contained in $\Omega_0$, which implies that there exist maximal elements $\Delta_0, \Delta_0^*\in \mathscr{U}_0$ such that $\mbox{axis}(g)$ intersects $\partial \Delta_0$ and $\partial \Delta_0^*$. Let $A,B$ denote the attracting and 
repelling fixed points of $g$. $\Delta_0$ and $\Delta_0^*$ are disjoint. Assume that $A\in \Delta_0\cap \mathbf{S}^1$ and $B\in \Delta_0^*\cap \mathbf{S}^1$. We know that $\Omega_0\subset \mathbf{H}\backslash \overline{\Delta_0\cup \Delta_0^*}$. Write $\overline{P_0Q_0}=\partial \Delta_0$. We refer to Figure 1, where $\Delta_0$ is the component of $\mathbf{H}\backslash \overline{P_0Q_0}$ containing $A$. 

Note that $\varrho(\mbox{axis}(g))=\tilde{c}$. The assumption that $i(\tilde{c},\tilde{u}_0)=N$, where $N\geq 2$, says that $\overline{P_1Q_1}=g(\partial \Delta_0^*)$ is disjoint from $\overline{P_0Q_0}$ and 
``lies below" $\overline{P_0Q_0}$. Let $R_0$ be the region bounded by $\overline{P_0Q_0}$ and $\overline{P_1Q_1}$. Observe that the geodesic $\mbox{axis}(g)$ inherits a natural orientation that points from $B$ to $A$. Now consider a point $z\in \mbox{axis}(g)$ moving  from $B$ to $A$ along $\mbox{axis}(g)$. When $z$
starts entering the region $R_0$, it crosses $N-1$ disjoint geodesics in $\{\varrho^{-1}(\tilde{u}_0)\}$ and then crosses $\overline{P_1Q_1}$ and leaves the region $R_0$. Of course, careful investigations on the $N-1$ geodesics and their relative positions are interesting but not needed in this paper. 

 For $j=1,\cdots, m$, we denote by 
\begin{equation}\label{K1}
\overline{P_{2j-1}Q_{2j-1}}=g^j(\partial \Delta_0^*) \ \ \mbox{and} \ \ \Delta_{2j-1}'=g^j(\Delta_0^*),
\end{equation}
and for $j=1,\cdots, m-1$, we let
\begin{equation}\label{K2}
\overline{P_{2j}Q_{2j}}=g^j(\overline{P_0Q_0}).
\end{equation} 
Let $\Delta_{2j}'$ be the component of  $\mathbf{H}\backslash \overline{P_{2j}Q_{2j}}$ containing the repelling fixed point $B$ of $g$. Then all $\overline{P_kQ_k}\in \{\varrho^{-1}(\tilde{u}_0)\}$; that is, $\varrho(\overline{P_kQ_k})=\tilde{u}_0$ for all $k=0,\cdots, 2m-1$.

\bigskip

\unitlength 1mm 
\linethickness{0.4pt}
\ifx\plotpoint\undefined\newsavebox{\plotpoint}\fi 
\begin{picture}(93.75,92)(0,0)
\put(90.919,54){\line(0,1){1.3336}}
\put(90.894,55.334){\line(0,1){1.3317}}
\multiput(90.819,56.665)(-.03135,.33198){4}{\line(0,1){.33198}}
\multiput(90.693,57.993)(-.029217,.220374){6}{\line(0,1){.220374}}
\multiput(90.518,59.315)(-.032138,.187815){7}{\line(0,1){.187815}}
\multiput(90.293,60.63)(-.030478,.145034){9}{\line(0,1){.145034}}
\multiput(90.019,61.935)(-.032326,.129405){10}{\line(0,1){.129405}}
\multiput(89.695,63.229)(-.030979,.106747){12}{\line(0,1){.106747}}
\multiput(89.324,64.51)(-.0322859,.097389){13}{\line(0,1){.097389}}
\multiput(88.904,65.776)(-.0333635,.0892397){14}{\line(0,1){.0892397}}
\multiput(88.437,67.026)(-.0321123,.0769302){16}{\line(0,1){.0769302}}
\multiput(87.923,68.257)(-.0329281,.0712156){17}{\line(0,1){.0712156}}
\multiput(87.363,69.467)(-.0336091,.0660406){18}{\line(0,1){.0660406}}
\multiput(86.758,70.656)(-.0324646,.0582555){20}{\line(0,1){.0582555}}
\multiput(86.109,71.821)(-.0329857,.0542779){21}{\line(0,1){.0542779}}
\multiput(85.416,72.961)(-.0334148,.0505885){22}{\line(0,1){.0505885}}
\multiput(84.681,74.074)(-.0323545,.0451867){24}{\line(0,1){.0451867}}
\multiput(83.905,75.159)(-.0326716,.042179){25}{\line(0,1){.042179}}
\multiput(83.088,76.213)(-.0329198,.0393451){26}{\line(0,1){.0393451}}
\multiput(82.232,77.236)(-.0331046,.0366675){27}{\line(0,1){.0366675}}
\multiput(81.338,78.226)(-.0332309,.0341309){28}{\line(0,1){.0341309}}
\multiput(80.408,79.182)(-.0344924,.0328555){28}{\line(-1,0){.0344924}}
\multiput(79.442,80.102)(-.0370274,.0327014){27}{\line(-1,0){.0370274}}
\multiput(78.442,80.985)(-.0397029,.0324873){26}{\line(-1,0){.0397029}}
\multiput(77.41,81.829)(-.0443061,.0335502){24}{\line(-1,0){.0443061}}
\multiput(76.346,82.634)(-.0475178,.0332434){23}{\line(-1,0){.0475178}}
\multiput(75.254,83.399)(-.0509511,.0328594){22}{\line(-1,0){.0509511}}
\multiput(74.133,84.122)(-.0546356,.0323899){21}{\line(-1,0){.0546356}}
\multiput(72.985,84.802)(-.0616918,.0335004){19}{\line(-1,0){.0616918}}
\multiput(71.813,85.439)(-.0664043,.0328846){18}{\line(-1,0){.0664043}}
\multiput(70.618,86.031)(-.0715716,.032147){17}{\line(-1,0){.0715716}}
\multiput(69.401,86.577)(-.0824287,.0333534){15}{\line(-1,0){.0824287}}
\multiput(68.165,87.077)(-.0895994,.0323852){14}{\line(-1,0){.0895994}}
\multiput(66.91,87.531)(-.0977364,.0312185){13}{\line(-1,0){.0977364}}
\multiput(65.64,87.937)(-.116814,.032519){11}{\line(-1,0){.116814}}
\multiput(64.355,88.294)(-.129751,.030908){10}{\line(-1,0){.129751}}
\multiput(63.057,88.603)(-.163528,.032501){8}{\line(-1,0){.163528}}
\multiput(61.749,88.863)(-.188156,.030081){7}{\line(-1,0){.188156}}
\multiput(60.432,89.074)(-.264817,.032165){5}{\line(-1,0){.264817}}
\put(59.108,89.235){\line(-1,0){1.3292}}
\put(57.779,89.346){\line(-1,0){1.3324}}
\put(56.446,89.406){\line(-1,0){1.3338}}
\put(55.113,89.417){\line(-1,0){1.3332}}
\put(53.779,89.377){\line(-1,0){1.3308}}
\multiput(52.449,89.287)(-.265291,-.027983){5}{\line(-1,0){.265291}}
\multiput(51.122,89.147)(-.220041,-.031626){6}{\line(-1,0){.220041}}
\multiput(49.802,88.958)(-.164021,-.029917){8}{\line(-1,0){.164021}}
\multiput(48.49,88.718)(-.144692,-.032063){9}{\line(-1,0){.144692}}
\multiput(47.187,88.43)(-.129044,-.033739){10}{\line(-1,0){.129044}}
\multiput(45.897,88.092)(-.106402,-.032145){12}{\line(-1,0){.106402}}
\multiput(44.62,87.707)(-.0970299,-.0333494){13}{\line(-1,0){.0970299}}
\multiput(43.359,87.273)(-.0829447,-.0320486){15}{\line(-1,0){.0829447}}
\multiput(42.115,86.792)(-.0765743,-.032952){16}{\line(-1,0){.0765743}}
\multiput(40.889,86.265)(-.0708511,-.0337052){17}{\line(-1,0){.0708511}}
\multiput(39.685,85.692)(-.0622127,-.0325228){19}{\line(-1,0){.0622127}}
\multiput(38.503,85.074)(-.0578968,-.0331){20}{\line(-1,0){.0578968}}
\multiput(37.345,84.412)(-.0539138,-.0335776){21}{\line(-1,0){.0539138}}
\multiput(36.213,83.707)(-.0480365,-.0324895){23}{\line(-1,0){.0480365}}
\multiput(35.108,82.96)(-.04483,-.0328469){24}{\line(-1,0){.04483}}
\multiput(34.032,82.171)(-.041819,-.0331311){25}{\line(-1,0){.041819}}
\multiput(32.987,81.343)(-.0389826,-.0333482){26}{\line(-1,0){.0389826}}
\multiput(31.973,80.476)(-.0363031,-.0335037){27}{\line(-1,0){.0363031}}
\multiput(30.993,79.572)(-.0337653,-.0336023){28}{\line(-1,0){.0337653}}
\multiput(30.047,78.631)(-.033679,-.0361405){27}{\line(0,-1){.0361405}}
\multiput(29.138,77.655)(-.0335365,-.0388208){26}{\line(0,-1){.0388208}}
\multiput(28.266,76.646)(-.0333331,-.0416582){25}{\line(0,-1){.0416582}}
\multiput(27.433,75.604)(-.0330635,-.0446705){24}{\line(0,-1){.0446705}}
\multiput(26.639,74.532)(-.0327215,-.0478787){23}{\line(0,-1){.0478787}}
\multiput(25.887,73.431)(-.0323,-.0513075){22}{\line(0,-1){.0513075}}
\multiput(25.176,72.302)(-.0333798,-.057736){20}{\line(0,-1){.057736}}
\multiput(24.508,71.147)(-.0328235,-.0620546){19}{\line(0,-1){.0620546}}
\multiput(23.885,69.968)(-.0321562,-.0667601){18}{\line(0,-1){.0667601}}
\multiput(23.306,68.767)(-.0333222,-.0764139){16}{\line(0,-1){.0764139}}
\multiput(22.773,67.544)(-.0324496,-.0827887){15}{\line(0,-1){.0827887}}
\multiput(22.286,66.302)(-.031403,-.0899483){14}{\line(0,-1){.0899483}}
\multiput(21.846,65.043)(-.03266,-.106245){12}{\line(0,-1){.106245}}
\multiput(21.455,63.768)(-.03124,-.117163){11}{\line(0,-1){.117163}}
\multiput(21.111,62.479)(-.032763,-.144535){9}{\line(0,-1){.144535}}
\multiput(20.816,61.178)(-.03071,-.163874){8}{\line(0,-1){.163874}}
\multiput(20.57,59.867)(-.032691,-.219886){6}{\line(0,-1){.219886}}
\multiput(20.374,58.548)(-.029266,-.265153){5}{\line(0,-1){.265153}}
\put(20.228,57.222){\line(0,-1){1.3303}}
\put(20.132,55.892){\line(0,-1){2.6668}}
\put(20.089,53.225){\line(0,-1){1.3327}}
\put(20.144,51.892){\line(0,-1){1.3297}}
\multiput(20.248,50.563)(.030884,-.264969){5}{\line(0,-1){.264969}}
\multiput(20.403,49.238)(.02917,-.188299){7}{\line(0,-1){.188299}}
\multiput(20.607,47.92)(.031709,-.163684){8}{\line(0,-1){.163684}}
\multiput(20.86,46.61)(.033644,-.144332){9}{\line(0,-1){.144332}}
\multiput(21.163,45.311)(.031954,-.11697){11}{\line(0,-1){.11697}}
\multiput(21.515,44.025)(.033307,-.106043){12}{\line(0,-1){.106043}}
\multiput(21.914,42.752)(.0319512,-.089755){14}{\line(0,-1){.089755}}
\multiput(22.362,41.495)(.0329541,-.0825892){15}{\line(0,-1){.0825892}}
\multiput(22.856,40.257)(.0318003,-.0717263){17}{\line(0,-1){.0717263}}
\multiput(23.397,39.037)(.0325629,-.0665627){18}{\line(0,-1){.0665627}}
\multiput(23.983,37.839)(.0332014,-.0618532){19}{\line(0,-1){.0618532}}
\multiput(24.614,36.664)(.0337314,-.0575313){20}{\line(0,-1){.0575313}}
\multiput(25.288,35.513)(.0326124,-.0511095){22}{\line(0,-1){.0511095}}
\multiput(26.006,34.389)(.033013,-.0476781){23}{\line(0,-1){.0476781}}
\multiput(26.765,33.292)(.0333354,-.044468){24}{\line(0,-1){.044468}}
\multiput(27.565,32.225)(.0335866,-.041454){25}{\line(0,-1){.041454}}
\multiput(28.405,31.189)(.0325219,-.0371852){27}{\line(0,-1){.0371852}}
\multiput(29.283,30.185)(.0326882,-.034651){28}{\line(0,-1){.034651}}
\multiput(30.198,29.215)(.0339697,-.0333957){28}{\line(1,0){.0339697}}
\multiput(31.149,28.279)(.0365068,-.0332816){27}{\line(1,0){.0365068}}
\multiput(32.135,27.381)(.0391853,-.0331098){26}{\line(1,0){.0391853}}
\multiput(33.154,26.52)(.0420204,-.0328753){25}{\line(1,0){.0420204}}
\multiput(34.204,25.698)(.0450296,-.0325728){24}{\line(1,0){.0450296}}
\multiput(35.285,24.916)(.0504262,-.0336592){22}{\line(1,0){.0504262}}
\multiput(36.394,24.176)(.0541177,-.033248){21}{\line(1,0){.0541177}}
\multiput(37.531,23.478)(.0580977,-.0327462){20}{\line(1,0){.0580977}}
\multiput(38.693,22.823)(.0624099,-.0321426){19}{\line(1,0){.0624099}}
\multiput(39.879,22.212)(.0710554,-.0332723){17}{\line(1,0){.0710554}}
\multiput(41.087,21.646)(.0767739,-.0324842){16}{\line(1,0){.0767739}}
\multiput(42.315,21.127)(.0831387,-.0315419){15}{\line(1,0){.0831387}}
\multiput(43.562,20.653)(.0972316,-.0327568){13}{\line(1,0){.0972316}}
\multiput(44.826,20.228)(.106596,-.031495){12}{\line(1,0){.106596}}
\multiput(46.105,19.85)(.129247,-.032952){10}{\line(1,0){.129247}}
\multiput(47.398,19.52)(.144885,-.03118){9}{\line(1,0){.144885}}
\multiput(48.702,19.24)(.187657,-.033046){7}{\line(1,0){.187657}}
\multiput(50.015,19.008)(.22023,-.030283){6}{\line(1,0){.22023}}
\multiput(51.337,18.827)(.33182,-.03295){4}{\line(1,0){.33182}}
\put(52.664,18.695){\line(1,0){1.3313}}
\put(53.995,18.613){\line(1,0){1.3334}}
\put(55.329,18.581){\line(1,0){1.3337}}
\put(56.662,18.6){\line(1,0){1.332}}
\put(57.994,18.669){\line(1,0){1.3285}}
\multiput(59.323,18.788)(.220513,.02815){6}{\line(1,0){.220513}}
\multiput(60.646,18.957)(.187969,.031228){7}{\line(1,0){.187969}}
\multiput(61.962,19.175)(.163327,.033498){8}{\line(1,0){.163327}}
\multiput(63.268,19.443)(.12956,.031699){10}{\line(1,0){.12956}}
\multiput(64.564,19.76)(.116613,.033232){11}{\line(1,0){.116613}}
\multiput(65.847,20.126)(.0975441,.0318143){13}{\line(1,0){.0975441}}
\multiput(67.115,20.54)(.0894001,.0329312){14}{\line(1,0){.0894001}}
\multiput(68.366,21.001)(.0770847,.0317397){16}{\line(1,0){.0770847}}
\multiput(69.6,21.508)(.0713741,.0325831){17}{\line(1,0){.0713741}}
\multiput(70.813,22.062)(.0662025,.0332891){18}{\line(1,0){.0662025}}
\multiput(72.005,22.662)(.0584119,.0321823){20}{\line(1,0){.0584119}}
\multiput(73.173,23.305)(.0544369,.0327227){21}{\line(1,0){.0544369}}
\multiput(74.316,23.992)(.0507496,.0331696){22}{\line(1,0){.0507496}}
\multiput(75.433,24.722)(.0473141,.0335327){23}{\line(1,0){.0473141}}
\multiput(76.521,25.493)(.0423366,.0324671){25}{\line(1,0){.0423366}}
\multiput(77.579,26.305)(.039504,.032729){26}{\line(1,0){.039504}}
\multiput(78.606,27.156)(.0368272,.0329267){27}{\line(1,0){.0368272}}
\multiput(79.601,28.045)(.0342913,.0330653){28}{\line(1,0){.0342913}}
\multiput(80.561,28.971)(.033022,.034333){28}{\line(0,1){.034333}}
\multiput(81.485,29.932)(.0328802,.0368687){27}{\line(0,1){.0368687}}
\multiput(82.373,30.928)(.0326791,.0395452){26}{\line(0,1){.0395452}}
\multiput(83.223,31.956)(.0324137,.0423775){25}{\line(0,1){.0423775}}
\multiput(84.033,33.015)(.0334729,.0473564){23}{\line(0,1){.0473564}}
\multiput(84.803,34.104)(.0331055,.0507914){22}{\line(0,1){.0507914}}
\multiput(85.531,35.222)(.032654,.0544782){21}{\line(0,1){.0544782}}
\multiput(86.217,36.366)(.0321086,.0584525){20}{\line(0,1){.0584525}}
\multiput(86.859,37.535)(.0332056,.0662444){18}{\line(0,1){.0662444}}
\multiput(87.457,38.727)(.032493,.0714152){17}{\line(0,1){.0714152}}
\multiput(88.009,39.941)(.0316424,.0771247){16}{\line(0,1){.0771247}}
\multiput(88.516,41.175)(.0328184,.0894416){14}{\line(0,1){.0894416}}
\multiput(88.975,42.428)(.0316912,.0975841){13}{\line(0,1){.0975841}}
\multiput(89.387,43.696)(.033084,.116655){11}{\line(0,1){.116655}}
\multiput(89.751,44.979)(.031536,.1296){10}{\line(0,1){.1296}}
\multiput(90.066,46.275)(.033292,.163369){8}{\line(0,1){.163369}}
\multiput(90.333,47.582)(.030991,.188008){7}{\line(0,1){.188008}}
\multiput(90.55,48.898)(.033447,.264658){5}{\line(0,1){.264658}}
\put(90.717,50.222){\line(0,1){1.3286}}
\put(90.834,51.55){\line(0,1){2.4497}}
\put(55.5,89.5){\line(0,-1){71}}
\qbezier(20.75,60.75)(54.625,54.25)(90,60.75)
\qbezier(20.25,55)(54,53.875)(90.75,55.25)
\qbezier(22.5,41.5)(55.125,51.875)(87.25,39.75)
\qbezier(24.5,37.25)(55.75,48.25)(85,35.25)
\qbezier(31.75,28.25)(55.125,42.25)(76,26.25)
\qbezier(35.25,25.25)(56,38.5)(72.75,23.75)
\qbezier(43.75,20.5)(56,32.625)(65.25,20.25)
\qbezier(47.5,19.75)(56.125,28.25)(62.25,19.75)
\qbezier(25.25,72.75)(55.125,59)(85.5,72.25)
\qbezier(28,76.25)(54.875,63.125)(83.25,75.5)
\qbezier(34.75,83)(55.25,64.875)(76.75,82.25)
\put(78.5,83.75){\makebox(0,0)[cc]{$Q_1$}}
\put(85.5,77.5){\makebox(0,0)[cc]{$Q_2$}}
\put(88,73.5){\makebox(0,0)[cc]{$Q_3$}}
\put(92.75,62){\makebox(0,0)[cc]{$Q_4$}}
\put(93.75,55.25){\makebox(0,0)[cc]{$Q_5$}}
\put(91.25,37.75){\makebox(0,0)[cc]{$Q_6$}}
\put(88,32.75){\makebox(0,0)[cc]{$Q_7$}}
\put(64,15){\makebox(0,0)[cc]{$Q_{2m-1}$}}
\put(70.25,17.25){\makebox(0,0)[cc]{$Q_{2m-2}$}}
\put(55.5,92){\makebox(0,0)[cc]{$B$}}
\put(55.25,15.25){\makebox(0,0)[cc]{$A$}}
\put(33.5,85){\makebox(0,0)[cc]{$P_1$}}
\put(26.5,77.75){\makebox(0,0)[cc]{$P_2$}}
\put(23.25,74){\makebox(0,0)[cc]{$P_3$}}
\put(18.25,61.75){\makebox(0,0)[cc]{$P_4$}}
\put(17.75,54.75){\makebox(0,0)[cc]{$P_5$}}
\put(19.75,40.5){\makebox(0,0)[cc]{$P_6$}}
\put(21.75,36){\makebox(0,0)[cc]{$P_7$}}
\put(39.5,19.25){\makebox(0,0)[cc]{$P_{2m-2}$}}
\put(47,15.75){\makebox(0,0)[cc]{$P_{2m-1}$}}
\put(32,77){\makebox(0,0)[cc]{$\Delta_2'$}}
\put(24.25,62.5){\makebox(0,0)[cc]{$\Delta_4'$}}
\put(24.75,45){\makebox(0,0)[cc]{$\Delta_6'$}}
\put(55,4.75){\makebox(0,0)[cc]{Fig. 1}}
\put(55.5,40.25){\vector(0,-1){2.25}}
\put(58,61.25){\makebox(0,0)[cc]{$g$}}
\put(87.75,57.5){\makebox(0,0)[cc]{$\Delta_5'$}}
\put(82.25,72.75){\makebox(0,0)[cc]{$\Delta_3'$}}
\put(38.25,87.75){\makebox(0,0)[cc]{$P_0$}}
\qbezier(45.75,88.25)(56,78.625)(66.25,87.5)
\qbezier(71.75,85.5)(55.75,69.25)(39.75,86)
\put(72,81.5){\makebox(0,0)[cc]{$\Delta_1'$}}
\put(73,87.75){\makebox(0,0)[cc]{$Q_0$}}
\put(51.5,86.75){\makebox(0,0)[cc]{$\Delta_0^*$}}
\put(83.75,38.5){\makebox(0,0)[cc]{$\Delta_7'$}}
\end{picture}

It is evident that all the geodesics $\overline{P_kQ_k}$, $0\leq k\leq 2m-1$, are mutually disjoint and for any $j=1,\cdots,m-1$, the geodesics $\overline{P_{2j}Q_{2j}}$ lies in between $\overline{P_{2j-1}Q_{2j-1}}$ 
and $\overline{P_{2j+1}Q_{2j+1}}$. The geodesics $\overline{P_kQ_k}$ with $1\leq k\leq 2m-1$ give rise to a partition of $\mathbf{H}$, and each one of which is referred to as a level geodesic with level $k$ in the sequel.

Let $(P_kP_{k+1})$ and $(Q_kQ_{k+1})$ denote the subarcs of $\mathbf{S}^1\backslash \{A,B\}$ connecting $P_k, P_{k+1}$ and  $Q_k, Q_{k+1}$, respectively. By examining the action of $g$ on $\mathbf{S}^1$, for $j=1,\cdots, m-2$, we have 
\begin{equation}\label{K3}
g(P_{2j-2}P_{2j})=(P_{2j}P_{2j+2}) \ \ \mbox{and}\ \ g(Q_{2j-2}Q_{2j})=(Q_{2j}Q_{2j+2}). 
\end{equation}

As usual, let $f^j(u_0)$ denote the geodesic homotopic to the image curve of $u_0$ under the map $f^j$ for all $j$. Set $u_m=f^m(u_0)$. Let $[u_0,u_1,\cdots, u_s,u_m]$ be a geodesic segment in $\mathcal{C}_1(S)$ joining $u_0$ and $u_m$. Then all $u_j$ are non-preperipheral and thus $\tilde{u}_j$ are all non-trivial simple closed geodesics on $\tilde{S}$. Let $(\tau_j, \Omega_j,\mathscr{U}_j)$ be the configurations corresponding to $u_j$. 

In what follows, the region $\Omega_j$ is said to be located above level $k$ for some $1\leq k\leq 2m-1$ if $\Omega_j\cap \Delta_k'\neq \emptyset$. Likewise, $\Omega_j$ is said to be located at level $k$ if $\Delta_k'$ is a maximal element of $\mathscr{U}_j$. See \cite{CZ10} for more detailed information.  

By construction, $\Omega_0\subset \mathbf{H}\backslash \overline{\Delta_0\cup \Delta_0^*}$ (in fact, $\mathbf{H}\backslash (\overline{\Omega_0}\cup \overline{\Delta_0\cup \Delta_0^*})$ is a disjoint union of infinitely many maximal elements of $\mathscr{U}_0$).  By a similar argument of Theorem 1.2 of \cite{CZ10} (together with Lemma 2.1 of \cite{CZ7}, (\ref{K2}) and (\ref{K3})), we know that $\Omega_1$ is located above or at level zero. By an induction argument,  one shows that for all $j=1,\cdots, s$ with $s\leq m-1$, $\Omega_j$ is located above or at level $2(j-1)$. In particular, we conclude that $\Omega_{m-1}$ is located above or at level $2(m-2)$. 

If $\Omega_{m-1}$ is located at level $2(m-2)=2m-4$, then there is a maximal element $\Delta_{m-1}\in \mathscr{U}_{m-1}$ that covers the attracting fixed point $A$ of $g$ such that $\partial \Delta_{m-1}$ lies above level $2(m-1)=2m-2$. Note that the point $P_{2m-2}$ lies in the arc $(P_{2m-3}P_{2m-1})$ and $Q_{2m-2}$ lies in the arc $(Q_{2m-3}Q_{2m-1})$. So the region bounded by the two geodesic 
$\overline{P_{2m-2}Q_{2m-2}}$ and $\overline{P_{2m-3}Q_{2m-3}}$ is not empty. By construction, $\Delta_{2m-1}'$ is the component of $\mathbf{H}\backslash \overline{P_{2m-1}Q_{2m-1}}$ containing the repelling fixed point $B$. Hence $\Delta_{m-1}\cap \Delta_{2m-1}'\neq \emptyset$ and $\Delta_{m-1}\cup \Delta_{2m-1}'=\mathbf{H}$.  Note also that the configuration 
$$
(\tau_m, \Omega_m,\mathscr{U}_m):=(g^m\tau_0g^{-m}, g^m(\Omega_0), g^m(\mathscr{U}_0))
$$ 
corresponds to $f^m(u_0)\in \mathcal{C}_0(S)$. Since $\Delta_0^*\in \mathscr{U}_0$, $g^m(\Delta_0^*)=\Delta_{2m-1}'\in g^m(\mathscr{U}_0)=\mathscr{U}_m$, we conclude that $\Delta_{2m-1}'\in \mathscr{U}_m$. By Lemma 4 of \cite{CZ1}, $u_{m-1}$ intersects $u_m=f^m(u_0)$, which implies that $d_{\mathcal{C}}(u_{m-1}, u_m)\geq 2$. Thus $s\geq m$, which says $d_{\mathcal{C}}(u_0, u_m)\geq m+1$, as asserted.  

If $\Omega_{m-1}$ is located above level $2(m-2)=2m-4$, then by Lemma 3.1 of \cite{CZ10}, there is a maximal element $\Delta_{m-1}\in \mathscr{U}_{m-1}$, which covers the attracting fixed point $A$, such that either (i) $\partial \Delta_{m-1}\cap \overline{P_{2m-2}Q_{2m-2}}\neq \emptyset$, or (ii) $\partial \Delta_{m-1}\cap \overline{P_{2m-2}Q_{2m-2}}= \emptyset$ but $\Delta_{m-1}\cup \Delta_{2m-1}'=\mathbf{H}$. If (ii) occurs, by Lemma 4 of \cite{CZ1} again, $u_{m-1}$ intersects $u_m=f^m(u_0)$, which implies that $d_{\mathcal{C}}(u_{m-1}, u_m)\geq 2$. So $s\geq m$. Suppose (i) occurs. 
We observe that $\varrho(\partial \Delta_{m-1})=\tilde{u}_{m-1}$ and $\varrho(\overline{P_{2m-2}Q_{2m-2}})=\tilde{u}_0$. 
Then $\tilde{u}_{m-1}$ intersects $\tilde{u}_{m}$. But $\tilde{u}_m=\tilde{u}_0$. This in turn implies that $u_{m-1}$ intersects $u_m$, and so $s\geq m$. 
\end{proof}


\unitlength 1mm 
\linethickness{0.4pt}
\ifx\plotpoint\undefined\newsavebox{\plotpoint}\fi 
\begin{picture}(92.5,81.75)(0,0)
\put(91.85,46){\line(0,1){1.2761}}
\put(91.825,47.276){\line(0,1){1.2742}}
\multiput(91.751,48.55)(-.03095,.31759){4}{\line(0,1){.31759}}
\multiput(91.627,49.821)(-.028842,.210765){6}{\line(0,1){.210765}}
\multiput(91.454,51.085)(-.031721,.179559){7}{\line(0,1){.179559}}
\multiput(91.232,52.342)(-.030079,.138593){9}{\line(0,1){.138593}}
\multiput(90.961,53.59)(-.031896,.123588){10}{\line(0,1){.123588}}
\multiput(90.642,54.825)(-.033339,.111141){11}{\line(0,1){.111141}}
\multiput(90.276,56.048)(-.0318419,.0928758){13}{\line(0,1){.0928758}}
\multiput(89.862,57.255)(-.0328954,.0850281){14}{\line(0,1){.0850281}}
\multiput(89.401,58.446)(-.031652,.0732253){16}{\line(0,1){.0732253}}
\multiput(88.895,59.617)(-.0324449,.0677086){17}{\line(0,1){.0677086}}
\multiput(88.343,60.768)(-.0331034,.0627084){18}{\line(0,1){.0627084}}
\multiput(87.747,61.897)(-.0336453,.0581448){19}{\line(0,1){.0581448}}
\multiput(87.108,63.002)(-.0324616,.051385){21}{\line(0,1){.051385}}
\multiput(86.426,64.081)(-.0328682,.0478085){22}{\line(0,1){.0478085}}
\multiput(85.703,65.133)(-.0331919,.0444741){23}{\line(0,1){.0444741}}
\multiput(84.94,66.156)(-.0334406,.0413531){24}{\line(0,1){.0413531}}
\multiput(84.137,67.148)(-.0336209,.0384219){25}{\line(0,1){.0384219}}
\multiput(83.297,68.109)(-.0337386,.0356604){26}{\line(0,1){.0356604}}
\multiput(82.419,69.036)(-.0337985,.0330516){27}{\line(-1,0){.0337985}}
\multiput(81.507,69.928)(-.0364053,.0329335){26}{\line(-1,0){.0364053}}
\multiput(80.56,70.785)(-.0391634,.0327541){25}{\line(-1,0){.0391634}}
\multiput(79.581,71.603)(-.0420899,.0325083){24}{\line(-1,0){.0420899}}
\multiput(78.571,72.384)(-.0472592,.0336532){22}{\line(-1,0){.0472592}}
\multiput(77.531,73.124)(-.0508419,.0333058){21}{\line(-1,0){.0508419}}
\multiput(76.464,73.823)(-.0547022,.0328708){20}{\line(-1,0){.0547022}}
\multiput(75.37,74.481)(-.058882,.0323379){19}{\line(-1,0){.058882}}
\multiput(74.251,75.095)(-.0671636,.0335585){17}{\line(-1,0){.0671636}}
\multiput(73.109,75.666)(-.0726926,.0328568){16}{\line(-1,0){.0726926}}
\multiput(71.946,76.191)(-.0788418,.0320087){15}{\line(-1,0){.0788418}}
\multiput(70.763,76.672)(-.0923373,.0333712){13}{\line(-1,0){.0923373}}
\multiput(69.563,77.105)(-.101361,.032239){12}{\line(-1,0){.101361}}
\multiput(68.347,77.492)(-.111858,.030848){11}{\line(-1,0){.111858}}
\multiput(67.116,77.832)(-.138077,.032363){9}{\line(-1,0){.138077}}
\multiput(65.874,78.123)(-.156634,.030347){8}{\line(-1,0){.156634}}
\multiput(64.621,78.366)(-.21026,.032319){6}{\line(-1,0){.21026}}
\multiput(63.359,78.56)(-.253628,.028951){5}{\line(-1,0){.253628}}
\put(62.091,78.704){\line(-1,0){1.2728}}
\put(60.818,78.8){\line(-1,0){1.2755}}
\put(59.542,78.846){\line(-1,0){1.2764}}
\put(58.266,78.842){\line(-1,0){1.2753}}
\put(56.991,78.789){\line(-1,0){1.2722}}
\multiput(55.719,78.686)(-.253455,-.030429){5}{\line(-1,0){.253455}}
\multiput(54.451,78.534)(-.210068,-.033543){6}{\line(-1,0){.210068}}
\multiput(53.191,78.332)(-.156455,-.031259){8}{\line(-1,0){.156455}}
\multiput(51.939,78.082)(-.137886,-.033167){9}{\line(-1,0){.137886}}
\multiput(50.698,77.784)(-.111677,-.031499){11}{\line(-1,0){.111677}}
\multiput(49.47,77.437)(-.101171,-.032829){12}{\line(-1,0){.101171}}
\multiput(48.256,77.043)(-.0855598,-.0314868){14}{\line(-1,0){.0855598}}
\multiput(47.058,76.602)(-.0786539,-.0324677){15}{\line(-1,0){.0786539}}
\multiput(45.878,76.115)(-.0724999,-.03328){16}{\line(-1,0){.0724999}}
\multiput(44.718,75.583)(-.0632465,-.0320634){18}{\line(-1,0){.0632465}}
\multiput(43.58,75.006)(-.0586925,-.0326805){19}{\line(-1,0){.0586925}}
\multiput(42.465,74.385)(-.0545097,-.0331891){20}{\line(-1,0){.0545097}}
\multiput(41.374,73.721)(-.0506469,-.0336015){21}{\line(-1,0){.0506469}}
\multiput(40.311,73.015)(-.0450161,-.0324529){23}{\line(-1,0){.0450161}}
\multiput(39.275,72.269)(-.0418997,-.0327531){24}{\line(-1,0){.0418997}}
\multiput(38.27,71.483)(-.0389719,-.0329818){25}{\line(-1,0){.0389719}}
\multiput(37.296,70.658)(-.0362127,-.0331451){26}{\line(-1,0){.0362127}}
\multiput(36.354,69.797)(-.0336053,-.0332481){27}{\line(-1,0){.0336053}}
\multiput(35.447,68.899)(-.0335301,-.0358565){26}{\line(0,-1){.0358565}}
\multiput(34.575,67.967)(-.0333964,-.0386172){25}{\line(0,-1){.0386172}}
\multiput(33.74,67.001)(-.033199,-.0415473){24}{\line(0,-1){.0415473}}
\multiput(32.943,66.004)(-.0329321,-.0446668){23}{\line(0,-1){.0446668}}
\multiput(32.186,64.977)(-.0325889,-.0479993){22}{\line(0,-1){.0479993}}
\multiput(31.469,63.921)(-.0321616,-.0515733){21}{\line(0,-1){.0515733}}
\multiput(30.793,62.838)(-.0333058,-.0583399){19}{\line(0,-1){.0583399}}
\multiput(30.161,61.729)(-.0327373,-.0629003){18}{\line(0,-1){.0629003}}
\multiput(29.571,60.597)(-.0320497,-.0678966){17}{\line(0,-1){.0678966}}
\multiput(29.026,59.443)(-.0333063,-.0783024){15}{\line(0,-1){.0783024}}
\multiput(28.527,58.268)(-.0323992,-.0852184){14}{\line(0,-1){.0852184}}
\multiput(28.073,57.075)(-.0313,-.0930598){13}{\line(0,-1){.0930598}}
\multiput(27.666,55.866)(-.032691,-.111334){11}{\line(0,-1){.111334}}
\multiput(27.307,54.641)(-.031175,-.123772){10}{\line(0,-1){.123772}}
\multiput(26.995,53.403)(-.032929,-.156112){8}{\line(0,-1){.156112}}
\multiput(26.732,52.154)(-.030674,-.179741){7}{\line(0,-1){.179741}}
\multiput(26.517,50.896)(-.033136,-.253115){5}{\line(0,-1){.253115}}
\put(26.351,49.63){\line(0,-1){1.2711}}
\put(26.235,48.359){\line(0,-1){1.2746}}
\put(26.168,47.085){\line(0,-1){1.2763}}
\put(26.151,45.809){\line(0,-1){1.276}}
\put(26.183,44.533){\line(0,-1){1.2738}}
\multiput(26.265,43.259)(.0328,-.3174){4}{\line(0,-1){.3174}}
\multiput(26.396,41.989)(.03007,-.210593){6}{\line(0,-1){.210593}}
\multiput(26.576,40.726)(.032767,-.179371){7}{\line(0,-1){.179371}}
\multiput(26.806,39.47)(.030886,-.138415){9}{\line(0,-1){.138415}}
\multiput(27.084,38.224)(.032616,-.1234){10}{\line(0,-1){.1234}}
\multiput(27.41,36.99)(.031154,-.1017){12}{\line(0,-1){.1017}}
\multiput(27.784,35.77)(.0323826,-.0926887){13}{\line(0,-1){.0926887}}
\multiput(28.205,34.565)(.0333904,-.084835){14}{\line(0,-1){.084835}}
\multiput(28.672,33.377)(.0320782,-.0730396){16}{\line(0,-1){.0730396}}
\multiput(29.185,32.209)(.0328389,-.0675184){17}{\line(0,-1){.0675184}}
\multiput(29.743,31.061)(.0334683,-.0625144){18}{\line(0,-1){.0625144}}
\multiput(30.346,29.936)(.0322844,-.0550504){20}{\line(0,-1){.0550504}}
\multiput(30.992,28.835)(.0327606,-.0511949){21}{\line(0,-1){.0511949}}
\multiput(31.68,27.759)(.0331462,-.0476161){22}{\line(0,-1){.0476161}}
\multiput(32.409,26.712)(.0334505,-.0442798){23}{\line(0,-1){.0442798}}
\multiput(33.178,25.693)(.033681,-.0411575){24}{\line(0,-1){.0411575}}
\multiput(33.987,24.706)(.0325426,-.0367551){26}{\line(0,-1){.0367551}}
\multiput(34.833,23.75)(.0326886,-.0341497){27}{\line(0,-1){.0341497}}
\multiput(35.715,22.828)(.0339905,-.0328541){27}{\line(1,0){.0339905}}
\multiput(36.633,21.941)(.0365966,-.0327207){26}{\line(1,0){.0365966}}
\multiput(37.584,21.09)(.0393537,-.0325253){25}{\line(1,0){.0393537}}
\multiput(38.568,20.277)(.0441169,-.0336652){23}{\line(1,0){.0441169}}
\multiput(39.583,19.503)(.0474546,-.0333771){22}{\line(1,0){.0474546}}
\multiput(40.627,18.768)(.0510352,-.0330089){21}{\line(1,0){.0510352}}
\multiput(41.699,18.075)(.0548929,-.0325514){20}{\line(1,0){.0548929}}
\multiput(42.797,17.424)(.0590695,-.0319941){19}{\line(1,0){.0590695}}
\multiput(43.919,16.816)(.0673581,-.0331665){17}{\line(1,0){.0673581}}
\multiput(45.064,16.253)(.0728829,-.0324326){16}{\line(1,0){.0728829}}
\multiput(46.23,15.734)(.079027,-.0315486){15}{\line(1,0){.079027}}
\multiput(47.416,15.26)(.0925303,-.0328325){13}{\line(1,0){.0925303}}
\multiput(48.618,14.834)(.101547,-.031648){12}{\line(1,0){.101547}}
\multiput(49.837,14.454)(.12324,-.033215){10}{\line(1,0){.12324}}
\multiput(51.069,14.122)(.138264,-.031558){9}{\line(1,0){.138264}}
\multiput(52.314,13.838)(.17921,-.033638){7}{\line(1,0){.17921}}
\multiput(53.568,13.602)(.210444,-.031092){6}{\line(1,0){.210444}}
\multiput(54.831,13.416)(.253792,-.027473){5}{\line(1,0){.253792}}
\put(56.1,13.278){\line(1,0){1.2733}}
\put(57.373,13.19){\line(1,0){1.2758}}
\put(58.649,13.152){\line(1,0){1.2763}}
\put(59.925,13.163){\line(1,0){1.2749}}
\put(61.2,13.224){\line(1,0){1.2716}}
\multiput(62.472,13.334)(.253273,.031906){5}{\line(1,0){.253273}}
\multiput(63.738,13.494)(.179887,.029801){7}{\line(1,0){.179887}}
\multiput(64.997,13.702)(.15627,.03217){8}{\line(1,0){.15627}}
\multiput(66.248,13.959)(.123922,.030573){10}{\line(1,0){.123922}}
\multiput(67.487,14.265)(.111491,.032149){11}{\line(1,0){.111491}}
\multiput(68.713,14.619)(.100978,.033418){12}{\line(1,0){.100978}}
\multiput(69.925,15.02)(.0853748,.0319849){14}{\line(1,0){.0853748}}
\multiput(71.12,15.468)(.0784633,.0329255){15}{\line(1,0){.0784633}}
\multiput(72.297,15.962)(.0723047,.033702){16}{\line(1,0){.0723047}}
\multiput(73.454,16.501)(.0630586,.0324314){18}{\line(1,0){.0630586}}
\multiput(74.589,17.085)(.058501,.0330221){19}{\line(1,0){.058501}}
\multiput(75.701,17.712)(.0543153,.0335063){20}{\line(1,0){.0543153}}
\multiput(76.787,18.382)(.048157,.0323554){22}{\line(1,0){.048157}}
\multiput(77.846,19.094)(.0448262,.0327147){23}{\line(1,0){.0448262}}
\multiput(78.877,19.846)(.0417081,.0329968){24}{\line(1,0){.0417081}}
\multiput(79.878,20.638)(.038779,.0332084){25}{\line(1,0){.038779}}
\multiput(80.848,21.468)(.0360189,.0333556){26}{\line(1,0){.0360189}}
\multiput(81.784,22.336)(.0334109,.0334434){27}{\line(0,1){.0334434}}
\multiput(82.686,23.239)(.0333206,.0360513){26}{\line(0,1){.0360513}}
\multiput(83.553,24.176)(.0331707,.0388112){25}{\line(0,1){.0388112}}
\multiput(84.382,25.146)(.0329562,.0417401){24}{\line(0,1){.0417401}}
\multiput(85.173,26.148)(.0326712,.044858){23}{\line(0,1){.044858}}
\multiput(85.924,27.18)(.0323086,.0481884){22}{\line(0,1){.0481884}}
\multiput(86.635,28.24)(.0334535,.0543479){20}{\line(0,1){.0543479}}
\multiput(87.304,29.327)(.0329652,.0585331){19}{\line(0,1){.0585331}}
\multiput(87.931,30.439)(.0323702,.06309){18}{\line(0,1){.06309}}
\multiput(88.513,31.575)(.0336317,.0723374){16}{\line(0,1){.0723374}}
\multiput(89.051,32.732)(.0328493,.0784952){15}{\line(0,1){.0784952}}
\multiput(89.544,33.909)(.031902,.0854058){14}{\line(0,1){.0854058}}
\multiput(89.991,35.105)(.03332,.101011){12}{\line(0,1){.101011}}
\multiput(90.391,36.317)(.032041,.111522){11}{\line(0,1){.111522}}
\multiput(90.743,37.544)(.030453,.123951){10}{\line(0,1){.123951}}
\multiput(91.048,38.784)(.032019,.156301){8}{\line(0,1){.156301}}
\multiput(91.304,40.034)(.029626,.179916){7}{\line(0,1){.179916}}
\multiput(91.511,41.293)(.03166,.253304){5}{\line(0,1){.253304}}
\put(91.669,42.56){\line(0,1){1.2717}}
\put(91.778,43.832){\line(0,1){2.1684}}
\qbezier(46.25,76.5)(59.5,61.75)(73.75,75)
\qbezier(30.5,62.75)(59.25,45.875)(88,61.5)
\qbezier(27.25,37.75)(59.625,45.625)(89.5,35)
\qbezier(34.5,24.25)(59.25,43.75)(82,23.25)
\qbezier(42.5,18)(60.375,39.25)(73.75,17.5)
\qbezier(49.5,14.75)(59.125,33.625)(68.25,15)
\put(28.25,64.25){\makebox(0,0)[cc]{$P_1$}}
\put(90.5,63){\makebox(0,0)[cc]{$Q_1$}}
\put(24.25,36.75){\makebox(0,0)[cc]{$P_2$}}
\put(92.5,33.5){\makebox(0,0)[cc]{$Q_2$}}
\put(32.25,22.5){\makebox(0,0)[cc]{$P_3$}}
\put(84,21){\makebox(0,0)[cc]{$Q_3$}}
\put(48.5,12){\makebox(0,0)[cc]{$P_m$}}
\put(69.5,12){\makebox(0,0)[cc]{$Q_m$}}
\put(40.5,15.5){\makebox(0,0)[cc]{$P_{m-1}$}}
\put(75.75,13.25){\makebox(0,0)[cc]{$Q_{m-1}$}}
\put(59.25,9.75){\makebox(0,0)[cc]{$A$}}
\put(59.5,81.75){\makebox(0,0)[cc]{$B$}}
\put(52.25,75.25){\makebox(0,0)[cc]{$\Delta_0^*$}}
\put(34.5,63.25){\makebox(0,0)[cc]{$\Delta_1'$}}
\put(29.25,41){\makebox(0,0)[cc]{$\Delta_2'$}}
\put(35,28.25){\makebox(0,0)[cc]{$\Delta_3'$}}
\put(48.5,18.25){\makebox(0,0)[cc]{$\Delta_m'$}}
\put(41,22.5){\makebox(0,0)[cc]{$\Delta_{m-1}'$}}
\put(62.25,47.75){\makebox(0,0)[cc]{$g$}}
\put(59.25,1.5){\makebox(0,0)[cc]{Fig. 2}}
\put(59.75,78.75){\vector(0,-1){65.75}}
\end{picture}

\section{Geodesic paths in the curve complex}
\setcounter{equation}{0}

In this section, we study geodesic segments connecting $u_0$ and $u_m$, where we recall that $u_m=f^m(u_0)$ which is the geodesic homotopic to the image curve of $u_0$ under the map $f^m$. For a discussion purpose, in what follows we only need a ``coarser partition" of $\mathbf{H}$ which is described below. See also \cite{CZ10} for more details. 

Let $\Delta_0, \Delta_0^*$ and $g$ be as in Section 3. For $j=1,\cdots, m$, write $\overline{P_jQ_j}=g^j(\partial \Delta_0^*)$. These geodesics $\overline{P_jQ_j}$ are referred to as level geodesics with level $j$. As usual,  put $\Delta_j'=g^j(\Delta_0^*)$. See Figure 2.  

Let $[u_0,u_1,\cdots, u_s,u_m]$ be a path connecting $u_0$ and $u_m$. Here we emphasize that the path is not assumed to be a geodesic segment. Then all $u_j$ are non-preperipheral. Once again, let $(\tau_j,\Omega_j, \mathscr{U}_j)$, $j=0,\cdots, s,m$, be the configurations corresponding to $u_j$. 

\begin{lem}\label{L3.1}
With the above notation, if $\Omega_j$ is located above level $j$ for some $j$ with $1\leq j\leq s$, then $s\geq m$. 
\end{lem}
\begin{proof}
If  $\Omega_j$ is located above level $j$ for some $j\leq m-2$, then by Lemma 3.1 of \cite{CZ10}, there exists a maximal element $\Delta_j\in \mathscr{U}_j$ such that $\Delta_j$ covers attracting fixed point $A$ of $g$ and 
either $\partial \Delta_j$ lies above $\overline{P_{j+1}Q_{j+1}}$ or $\partial \Delta_j$ crosses $\overline{P_{j+1}Q_{j+1}}$ (Figure 3 and Figure 4).

\bigskip

\unitlength 1mm 
\linethickness{0.4pt}
\ifx\plotpoint\undefined\newsavebox{\plotpoint}\fi 
\begin{picture}(114.637,68.75)(0,0)
\put(50.887,40.75){\line(0,1){1.0864}}
\put(50.864,41.836){\line(0,1){1.0845}}
\put(50.794,42.921){\line(0,1){1.0805}}
\multiput(50.678,44.001)(-.032442,.214905){5}{\line(0,1){.214905}}
\multiput(50.516,45.076)(-.029721,.152371){7}{\line(0,1){.152371}}
\multiput(50.308,46.143)(-.031688,.13209){8}{\line(0,1){.13209}}
\multiput(50.054,47.199)(-.033166,.1161){9}{\line(0,1){.1161}}
\multiput(49.756,48.244)(-.031176,.093743){11}{\line(0,1){.093743}}
\multiput(49.413,49.275)(-.032229,.084629){12}{\line(0,1){.084629}}
\multiput(49.026,50.291)(-.033066,.0767743){13}{\line(0,1){.0767743}}
\multiput(48.596,51.289)(-.0337269,.0699111){14}{\line(0,1){.0699111}}
\multiput(48.124,52.268)(-.0321019,.0598532){16}{\line(0,1){.0598532}}
\multiput(47.61,53.225)(-.0325967,.0549878){17}{\line(0,1){.0549878}}
\multiput(47.056,54.16)(-.03298,.0505679){18}{\line(0,1){.0505679}}
\multiput(46.463,55.07)(-.0332658,.0465254){19}{\line(0,1){.0465254}}
\multiput(45.83,55.954)(-.0334651,.0428062){20}{\line(0,1){.0428062}}
\multiput(45.161,56.81)(-.033587,.0393665){21}{\line(0,1){.0393665}}
\multiput(44.456,57.637)(-.0336391,.0361707){22}{\line(0,1){.0361707}}
\multiput(43.716,58.433)(-.0336277,.0331893){23}{\line(-1,0){.0336277}}
\multiput(42.942,59.196)(-.0366089,.0331616){22}{\line(-1,0){.0366089}}
\multiput(42.137,59.926)(-.0398038,.0330676){21}{\line(-1,0){.0398038}}
\multiput(41.301,60.62)(-.0432416,.0329006){20}{\line(-1,0){.0432416}}
\multiput(40.436,61.278)(-.0469579,.0326525){19}{\line(-1,0){.0469579}}
\multiput(39.544,61.899)(-.0509962,.0323137){18}{\line(-1,0){.0509962}}
\multiput(38.626,62.48)(-.0554108,.0318724){17}{\line(-1,0){.0554108}}
\multiput(37.684,63.022)(-.0642872,.0334015){15}{\line(-1,0){.0642872}}
\multiput(36.72,63.523)(-.0703476,.0328068){14}{\line(-1,0){.0703476}}
\multiput(35.735,63.982)(-.0772015,.0320559){13}{\line(-1,0){.0772015}}
\multiput(34.731,64.399)(-.085045,.031116){12}{\line(-1,0){.085045}}
\multiput(33.711,64.773)(-.103558,.032938){10}{\line(-1,0){.103558}}
\multiput(32.675,65.102)(-.116525,.03164){9}{\line(-1,0){.116525}}
\multiput(31.627,65.387)(-.132494,.029952){8}{\line(-1,0){.132494}}
\multiput(30.567,65.626)(-.178206,.032339){6}{\line(-1,0){.178206}}
\multiput(29.497,65.82)(-.215312,.02962){5}{\line(-1,0){.215312}}
\put(28.421,65.968){\line(-1,0){1.0819}}
\put(27.339,66.07){\line(-1,0){1.0853}}
\put(26.254,66.126){\line(-1,0){1.0867}}
\put(25.167,66.135){\line(-1,0){1.0861}}
\put(24.081,66.097){\line(-1,0){1.0835}}
\multiput(22.997,66.013)(-.26972,-.03256){4}{\line(-1,0){.26972}}
\multiput(21.919,65.883)(-.178717,-.029383){6}{\line(-1,0){.178717}}
\multiput(20.846,65.707)(-.151968,-.031718){7}{\line(-1,0){.151968}}
\multiput(19.782,65.485)(-.131662,-.033418){8}{\line(-1,0){.131662}}
\multiput(18.729,65.217)(-.104089,-.031218){10}{\line(-1,0){.104089}}
\multiput(17.688,64.905)(-.093325,-.032403){11}{\line(-1,0){.093325}}
\multiput(16.662,64.549)(-.084199,-.033337){12}{\line(-1,0){.084199}}
\multiput(15.651,64.149)(-.0708814,-.0316369){14}{\line(-1,0){.0708814}}
\multiput(14.659,63.706)(-.0648317,-.0323319){15}{\line(-1,0){.0648317}}
\multiput(13.686,63.221)(-.0594269,-.0328845){16}{\line(-1,0){.0594269}}
\multiput(12.736,62.695)(-.0545554,-.0333153){17}{\line(-1,0){.0545554}}
\multiput(11.808,62.128)(-.0501308,-.0336407){18}{\line(-1,0){.0501308}}
\multiput(10.906,61.523)(-.0437807,-.0321797){20}{\line(-1,0){.0437807}}
\multiput(10.03,60.879)(-.0403461,-.0324037){21}{\line(-1,0){.0403461}}
\multiput(9.183,60.199)(-.0371532,-.0325506){22}{\line(-1,0){.0371532}}
\multiput(8.366,59.483)(-.0341729,-.0326277){23}{\line(-1,0){.0341729}}
\multiput(7.58,58.732)(-.0327452,-.0340602){23}{\line(0,-1){.0340602}}
\multiput(6.826,57.949)(-.0326784,-.0370408){22}{\line(0,-1){.0370408}}
\multiput(6.108,57.134)(-.0325425,-.0402342){21}{\line(0,-1){.0402342}}
\multiput(5.424,56.289)(-.0323304,-.0436695){20}{\line(0,-1){.0436695}}
\multiput(4.778,55.416)(-.0320336,-.0473822){19}{\line(0,-1){.0473822}}
\multiput(4.169,54.515)(-.0335031,-.0544403){17}{\line(0,-1){.0544403}}
\multiput(3.599,53.59)(-.0330891,-.0593132){16}{\line(0,-1){.0593132}}
\multiput(3.07,52.641)(-.0325551,-.0647199){15}{\line(0,-1){.0647199}}
\multiput(2.582,51.67)(-.031881,-.070772){14}{\line(0,-1){.070772}}
\multiput(2.135,50.679)(-.033627,-.084083){12}{\line(0,-1){.084083}}
\multiput(1.732,49.67)(-.032725,-.093213){11}{\line(0,-1){.093213}}
\multiput(1.372,48.645)(-.031576,-.103981){10}{\line(0,-1){.103981}}
\multiput(1.056,47.605)(-.030108,-.11693){9}{\line(0,-1){.11693}}
\multiput(.785,46.553)(-.032241,-.151858){7}{\line(0,-1){.151858}}
\multiput(.559,45.49)(-.029998,-.178615){6}{\line(0,-1){.178615}}
\multiput(.379,44.418)(-.03349,-.2696){4}{\line(0,-1){.2696}}
\put(.245,43.34){\line(0,-1){1.0832}}
\put(.158,42.256){\line(0,-1){2.1726}}
\put(.122,40.084){\line(0,-1){1.0855}}
\put(.173,38.998){\line(0,-1){1.0823}}
\multiput(.272,37.916)(.028878,-.215413){5}{\line(0,-1){.215413}}
\multiput(.416,36.839)(.031725,-.178316){6}{\line(0,-1){.178316}}
\multiput(.606,35.769)(.033709,-.151539){7}{\line(0,-1){.151539}}
\multiput(.842,34.708)(.031238,-.116633){9}{\line(0,-1){.116633}}
\multiput(1.124,33.659)(.032581,-.103671){10}{\line(0,-1){.103671}}
\multiput(1.449,32.622)(.033625,-.092892){11}{\line(0,-1){.092892}}
\multiput(1.819,31.6)(.0317896,-.0773115){13}{\line(0,-1){.0773115}}
\multiput(2.232,30.595)(.0325642,-.0704602){14}{\line(0,-1){.0704602}}
\multiput(2.688,29.609)(.0331797,-.0644019){15}{\line(0,-1){.0644019}}
\multiput(3.186,28.643)(.0336613,-.0589903){16}{\line(0,-1){.0589903}}
\multiput(3.725,27.699)(.0321378,-.0511073){18}{\line(0,-1){.0511073}}
\multiput(4.303,26.779)(.0324905,-.0470701){19}{\line(0,-1){.0470701}}
\multiput(4.92,25.885)(.0327514,-.0433547){20}{\line(0,-1){.0433547}}
\multiput(5.575,25.017)(.0329302,-.0399175){21}{\line(0,-1){.0399175}}
\multiput(6.267,24.179)(.0330352,-.036723){22}{\line(0,-1){.036723}}
\multiput(6.994,23.371)(.0330732,-.0337418){23}{\line(0,-1){.0337418}}
\multiput(7.754,22.595)(.0344869,-.0322955){23}{\line(1,0){.0344869}}
\multiput(8.548,21.852)(.0392505,-.0337225){21}{\line(1,0){.0392505}}
\multiput(9.372,21.144)(.0426906,-.0336124){20}{\line(1,0){.0426906}}
\multiput(10.226,20.472)(.0464105,-.0334259){19}{\line(1,0){.0464105}}
\multiput(11.108,19.837)(.0504539,-.0331541){18}{\line(1,0){.0504539}}
\multiput(12.016,19.24)(.0548752,-.032786){17}{\line(1,0){.0548752}}
\multiput(12.949,18.683)(.0597422,-.032308){16}{\line(1,0){.0597422}}
\multiput(13.904,18.166)(.0651415,-.0317031){15}{\line(1,0){.0651415}}
\multiput(14.882,17.69)(.0766599,-.0333304){13}{\line(1,0){.0766599}}
\multiput(15.878,17.257)(.084517,-.032521){12}{\line(1,0){.084517}}
\multiput(16.892,16.867)(.093635,-.031499){11}{\line(1,0){.093635}}
\multiput(17.922,16.52)(.115985,-.033566){9}{\line(1,0){.115985}}
\multiput(18.966,16.218)(.13198,-.032143){8}{\line(1,0){.13198}}
\multiput(20.022,15.961)(.152268,-.030246){7}{\line(1,0){.152268}}
\multiput(21.088,15.749)(.214792,-.033183){5}{\line(1,0){.214792}}
\put(22.162,15.583){\line(1,0){1.0801}}
\put(23.242,15.464){\line(1,0){1.0842}}
\put(24.326,15.39){\line(1,0){2.1729}}
\put(26.499,15.383){\line(1,0){1.0847}}
\put(27.584,15.449){\line(1,0){1.0809}}
\multiput(28.665,15.561)(.215015,.031702){5}{\line(1,0){.215015}}
\multiput(29.74,15.72)(.152473,.029196){7}{\line(1,0){.152473}}
\multiput(30.807,15.924)(.132198,.031232){8}{\line(1,0){.132198}}
\multiput(31.865,16.174)(.116214,.032766){9}{\line(1,0){.116214}}
\multiput(32.91,16.469)(.093849,.030853){11}{\line(1,0){.093849}}
\multiput(33.943,16.808)(.084739,.031937){12}{\line(1,0){.084739}}
\multiput(34.96,17.191)(.0768878,.0328013){13}{\line(1,0){.0768878}}
\multiput(35.959,17.618)(.0700269,.0334858){14}{\line(1,0){.0700269}}
\multiput(36.94,18.086)(.0599635,.0318955){16}{\line(1,0){.0599635}}
\multiput(37.899,18.597)(.0550998,.032407){17}{\line(1,0){.0550998}}
\multiput(38.836,19.148)(.0506812,.0328056){18}{\line(1,0){.0506812}}
\multiput(39.748,19.738)(.0466398,.0331053){19}{\line(1,0){.0466398}}
\multiput(40.634,20.367)(.0429213,.0333174){20}{\line(1,0){.0429213}}
\multiput(41.493,21.034)(.039482,.0334511){21}{\line(1,0){.039482}}
\multiput(42.322,21.736)(.0362864,.0335142){22}{\line(1,0){.0362864}}
\multiput(43.12,22.473)(.033305,.0335131){23}{\line(0,1){.0335131}}
\multiput(43.886,23.244)(.0332876,.0364944){22}{\line(0,1){.0364944}}
\multiput(44.618,24.047)(.0332046,.0396896){21}{\line(0,1){.0396896}}
\multiput(45.316,24.881)(.0330494,.043128){20}{\line(0,1){.043128}}
\multiput(45.977,25.743)(.0328141,.0468451){19}{\line(0,1){.0468451}}
\multiput(46.6,26.633)(.0324893,.0508846){18}{\line(0,1){.0508846}}
\multiput(47.185,27.549)(.0320632,.0553006){17}{\line(0,1){.0553006}}
\multiput(47.73,28.489)(.0336228,.0641717){15}{\line(0,1){.0641717}}
\multiput(48.234,29.452)(.033049,.0702341){14}{\line(0,1){.0702341}}
\multiput(48.697,30.435)(.0323217,.0770906){13}{\line(0,1){.0770906}}
\multiput(49.117,31.437)(.031409,.084937){12}{\line(0,1){.084937}}
\multiput(49.494,32.456)(.033294,.103444){10}{\line(0,1){.103444}}
\multiput(49.827,33.491)(.032041,.116415){9}{\line(0,1){.116415}}
\multiput(50.115,34.539)(.030408,.13239){8}{\line(0,1){.13239}}
\multiput(50.359,35.598)(.032953,.178093){6}{\line(0,1){.178093}}
\multiput(50.556,36.666)(.030362,.215208){5}{\line(0,1){.215208}}
\put(50.708,37.742){\line(0,1){1.0816}}
\put(50.814,38.824){\line(0,1){1.9261}}
\put(114.637,40.75){\line(0,1){1.0864}}
\put(114.614,41.836){\line(0,1){1.0845}}
\put(114.544,42.921){\line(0,1){1.0805}}
\multiput(114.428,44.001)(-.032442,.214905){5}{\line(0,1){.214905}}
\multiput(114.266,45.076)(-.029721,.152371){7}{\line(0,1){.152371}}
\multiput(114.058,46.143)(-.031688,.13209){8}{\line(0,1){.13209}}
\multiput(113.804,47.199)(-.033166,.1161){9}{\line(0,1){.1161}}
\multiput(113.506,48.244)(-.031176,.093743){11}{\line(0,1){.093743}}
\multiput(113.163,49.275)(-.032229,.084629){12}{\line(0,1){.084629}}
\multiput(112.776,50.291)(-.033066,.0767743){13}{\line(0,1){.0767743}}
\multiput(112.346,51.289)(-.0337269,.0699111){14}{\line(0,1){.0699111}}
\multiput(111.874,52.268)(-.0321019,.0598532){16}{\line(0,1){.0598532}}
\multiput(111.36,53.225)(-.0325967,.0549878){17}{\line(0,1){.0549878}}
\multiput(110.806,54.16)(-.03298,.0505679){18}{\line(0,1){.0505679}}
\multiput(110.213,55.07)(-.0332658,.0465254){19}{\line(0,1){.0465254}}
\multiput(109.58,55.954)(-.0334651,.0428062){20}{\line(0,1){.0428062}}
\multiput(108.911,56.81)(-.033587,.0393665){21}{\line(0,1){.0393665}}
\multiput(108.206,57.637)(-.0336391,.0361707){22}{\line(0,1){.0361707}}
\multiput(107.466,58.433)(-.0336277,.0331893){23}{\line(-1,0){.0336277}}
\multiput(106.692,59.196)(-.0366089,.0331616){22}{\line(-1,0){.0366089}}
\multiput(105.887,59.926)(-.0398038,.0330676){21}{\line(-1,0){.0398038}}
\multiput(105.051,60.62)(-.0432416,.0329006){20}{\line(-1,0){.0432416}}
\multiput(104.186,61.278)(-.0469579,.0326525){19}{\line(-1,0){.0469579}}
\multiput(103.294,61.899)(-.0509962,.0323137){18}{\line(-1,0){.0509962}}
\multiput(102.376,62.48)(-.0554108,.0318724){17}{\line(-1,0){.0554108}}
\multiput(101.434,63.022)(-.0642872,.0334015){15}{\line(-1,0){.0642872}}
\multiput(100.47,63.523)(-.0703476,.0328068){14}{\line(-1,0){.0703476}}
\multiput(99.485,63.982)(-.0772015,.0320559){13}{\line(-1,0){.0772015}}
\multiput(98.481,64.399)(-.085045,.031116){12}{\line(-1,0){.085045}}
\multiput(97.461,64.773)(-.103558,.032938){10}{\line(-1,0){.103558}}
\multiput(96.425,65.102)(-.116525,.03164){9}{\line(-1,0){.116525}}
\multiput(95.377,65.387)(-.132494,.029952){8}{\line(-1,0){.132494}}
\multiput(94.317,65.626)(-.178206,.032339){6}{\line(-1,0){.178206}}
\multiput(93.247,65.82)(-.215312,.02962){5}{\line(-1,0){.215312}}
\put(92.171,65.968){\line(-1,0){1.0819}}
\put(91.089,66.07){\line(-1,0){1.0853}}
\put(90.004,66.126){\line(-1,0){1.0867}}
\put(88.917,66.135){\line(-1,0){1.0861}}
\put(87.831,66.097){\line(-1,0){1.0835}}
\multiput(86.747,66.013)(-.26972,-.03256){4}{\line(-1,0){.26972}}
\multiput(85.669,65.883)(-.178717,-.029383){6}{\line(-1,0){.178717}}
\multiput(84.596,65.707)(-.151968,-.031718){7}{\line(-1,0){.151968}}
\multiput(83.532,65.485)(-.131662,-.033418){8}{\line(-1,0){.131662}}
\multiput(82.479,65.217)(-.104089,-.031218){10}{\line(-1,0){.104089}}
\multiput(81.438,64.905)(-.093325,-.032403){11}{\line(-1,0){.093325}}
\multiput(80.412,64.549)(-.084199,-.033337){12}{\line(-1,0){.084199}}
\multiput(79.401,64.149)(-.0708814,-.0316369){14}{\line(-1,0){.0708814}}
\multiput(78.409,63.706)(-.0648317,-.0323319){15}{\line(-1,0){.0648317}}
\multiput(77.436,63.221)(-.0594269,-.0328845){16}{\line(-1,0){.0594269}}
\multiput(76.486,62.695)(-.0545554,-.0333153){17}{\line(-1,0){.0545554}}
\multiput(75.558,62.128)(-.0501308,-.0336407){18}{\line(-1,0){.0501308}}
\multiput(74.656,61.523)(-.0437807,-.0321797){20}{\line(-1,0){.0437807}}
\multiput(73.78,60.879)(-.0403461,-.0324037){21}{\line(-1,0){.0403461}}
\multiput(72.933,60.199)(-.0371532,-.0325506){22}{\line(-1,0){.0371532}}
\multiput(72.116,59.483)(-.0341729,-.0326277){23}{\line(-1,0){.0341729}}
\multiput(71.33,58.732)(-.0327452,-.0340602){23}{\line(0,-1){.0340602}}
\multiput(70.576,57.949)(-.0326784,-.0370408){22}{\line(0,-1){.0370408}}
\multiput(69.858,57.134)(-.0325425,-.0402342){21}{\line(0,-1){.0402342}}
\multiput(69.174,56.289)(-.0323304,-.0436695){20}{\line(0,-1){.0436695}}
\multiput(68.528,55.416)(-.0320336,-.0473822){19}{\line(0,-1){.0473822}}
\multiput(67.919,54.515)(-.0335031,-.0544403){17}{\line(0,-1){.0544403}}
\multiput(67.349,53.59)(-.0330891,-.0593132){16}{\line(0,-1){.0593132}}
\multiput(66.82,52.641)(-.0325551,-.0647199){15}{\line(0,-1){.0647199}}
\multiput(66.332,51.67)(-.031881,-.070772){14}{\line(0,-1){.070772}}
\multiput(65.885,50.679)(-.033627,-.084083){12}{\line(0,-1){.084083}}
\multiput(65.482,49.67)(-.032725,-.093213){11}{\line(0,-1){.093213}}
\multiput(65.122,48.645)(-.031576,-.103981){10}{\line(0,-1){.103981}}
\multiput(64.806,47.605)(-.030108,-.11693){9}{\line(0,-1){.11693}}
\multiput(64.535,46.553)(-.032241,-.151858){7}{\line(0,-1){.151858}}
\multiput(64.309,45.49)(-.029998,-.178615){6}{\line(0,-1){.178615}}
\multiput(64.129,44.418)(-.03349,-.2696){4}{\line(0,-1){.2696}}
\put(63.995,43.34){\line(0,-1){1.0832}}
\put(63.908,42.256){\line(0,-1){2.1726}}
\put(63.872,40.084){\line(0,-1){1.0855}}
\put(63.923,38.998){\line(0,-1){1.0823}}
\multiput(64.022,37.916)(.028878,-.215413){5}{\line(0,-1){.215413}}
\multiput(64.166,36.839)(.031725,-.178316){6}{\line(0,-1){.178316}}
\multiput(64.356,35.769)(.033709,-.151539){7}{\line(0,-1){.151539}}
\multiput(64.592,34.708)(.031238,-.116633){9}{\line(0,-1){.116633}}
\multiput(64.874,33.659)(.032581,-.103671){10}{\line(0,-1){.103671}}
\multiput(65.199,32.622)(.033625,-.092892){11}{\line(0,-1){.092892}}
\multiput(65.569,31.6)(.0317896,-.0773115){13}{\line(0,-1){.0773115}}
\multiput(65.982,30.595)(.0325642,-.0704602){14}{\line(0,-1){.0704602}}
\multiput(66.438,29.609)(.0331797,-.0644019){15}{\line(0,-1){.0644019}}
\multiput(66.936,28.643)(.0336613,-.0589903){16}{\line(0,-1){.0589903}}
\multiput(67.475,27.699)(.0321378,-.0511073){18}{\line(0,-1){.0511073}}
\multiput(68.053,26.779)(.0324905,-.0470701){19}{\line(0,-1){.0470701}}
\multiput(68.67,25.885)(.0327514,-.0433547){20}{\line(0,-1){.0433547}}
\multiput(69.325,25.017)(.0329302,-.0399175){21}{\line(0,-1){.0399175}}
\multiput(70.017,24.179)(.0330352,-.036723){22}{\line(0,-1){.036723}}
\multiput(70.744,23.371)(.0330732,-.0337418){23}{\line(0,-1){.0337418}}
\multiput(71.504,22.595)(.0344869,-.0322955){23}{\line(1,0){.0344869}}
\multiput(72.298,21.852)(.0392505,-.0337225){21}{\line(1,0){.0392505}}
\multiput(73.122,21.144)(.0426906,-.0336124){20}{\line(1,0){.0426906}}
\multiput(73.976,20.472)(.0464105,-.0334259){19}{\line(1,0){.0464105}}
\multiput(74.858,19.837)(.0504539,-.0331541){18}{\line(1,0){.0504539}}
\multiput(75.766,19.24)(.0548752,-.032786){17}{\line(1,0){.0548752}}
\multiput(76.699,18.683)(.0597422,-.032308){16}{\line(1,0){.0597422}}
\multiput(77.654,18.166)(.0651415,-.0317031){15}{\line(1,0){.0651415}}
\multiput(78.632,17.69)(.0766599,-.0333304){13}{\line(1,0){.0766599}}
\multiput(79.628,17.257)(.084517,-.032521){12}{\line(1,0){.084517}}
\multiput(80.642,16.867)(.093635,-.031499){11}{\line(1,0){.093635}}
\multiput(81.672,16.52)(.115985,-.033566){9}{\line(1,0){.115985}}
\multiput(82.716,16.218)(.13198,-.032143){8}{\line(1,0){.13198}}
\multiput(83.772,15.961)(.152268,-.030246){7}{\line(1,0){.152268}}
\multiput(84.838,15.749)(.214792,-.033183){5}{\line(1,0){.214792}}
\put(85.912,15.583){\line(1,0){1.0801}}
\put(86.992,15.464){\line(1,0){1.0842}}
\put(88.076,15.39){\line(1,0){2.1729}}
\put(90.249,15.383){\line(1,0){1.0847}}
\put(91.334,15.449){\line(1,0){1.0809}}
\multiput(92.415,15.561)(.215015,.031702){5}{\line(1,0){.215015}}
\multiput(93.49,15.72)(.152473,.029196){7}{\line(1,0){.152473}}
\multiput(94.557,15.924)(.132198,.031232){8}{\line(1,0){.132198}}
\multiput(95.615,16.174)(.116214,.032766){9}{\line(1,0){.116214}}
\multiput(96.66,16.469)(.093849,.030853){11}{\line(1,0){.093849}}
\multiput(97.693,16.808)(.084739,.031937){12}{\line(1,0){.084739}}
\multiput(98.71,17.191)(.0768878,.0328013){13}{\line(1,0){.0768878}}
\multiput(99.709,17.618)(.0700269,.0334858){14}{\line(1,0){.0700269}}
\multiput(100.69,18.086)(.0599635,.0318955){16}{\line(1,0){.0599635}}
\multiput(101.649,18.597)(.0550998,.032407){17}{\line(1,0){.0550998}}
\multiput(102.586,19.148)(.0506812,.0328056){18}{\line(1,0){.0506812}}
\multiput(103.498,19.738)(.0466398,.0331053){19}{\line(1,0){.0466398}}
\multiput(104.384,20.367)(.0429213,.0333174){20}{\line(1,0){.0429213}}
\multiput(105.243,21.034)(.039482,.0334511){21}{\line(1,0){.039482}}
\multiput(106.072,21.736)(.0362864,.0335142){22}{\line(1,0){.0362864}}
\multiput(106.87,22.473)(.033305,.0335131){23}{\line(0,1){.0335131}}
\multiput(107.636,23.244)(.0332876,.0364944){22}{\line(0,1){.0364944}}
\multiput(108.368,24.047)(.0332046,.0396896){21}{\line(0,1){.0396896}}
\multiput(109.066,24.881)(.0330494,.043128){20}{\line(0,1){.043128}}
\multiput(109.727,25.743)(.0328141,.0468451){19}{\line(0,1){.0468451}}
\multiput(110.35,26.633)(.0324893,.0508846){18}{\line(0,1){.0508846}}
\multiput(110.935,27.549)(.0320632,.0553006){17}{\line(0,1){.0553006}}
\multiput(111.48,28.489)(.0336228,.0641717){15}{\line(0,1){.0641717}}
\multiput(111.984,29.452)(.033049,.0702341){14}{\line(0,1){.0702341}}
\multiput(112.447,30.435)(.0323217,.0770906){13}{\line(0,1){.0770906}}
\multiput(112.867,31.437)(.031409,.084937){12}{\line(0,1){.084937}}
\multiput(113.244,32.456)(.033294,.103444){10}{\line(0,1){.103444}}
\multiput(113.577,33.491)(.032041,.116415){9}{\line(0,1){.116415}}
\multiput(113.865,34.539)(.030408,.13239){8}{\line(0,1){.13239}}
\multiput(114.109,35.598)(.032953,.178093){6}{\line(0,1){.178093}}
\multiput(114.306,36.666)(.030362,.215208){5}{\line(0,1){.215208}}
\put(114.458,37.742){\line(0,1){1.0816}}
\put(114.564,38.824){\line(0,1){1.9261}}
\qbezier(5.25,25.5)(25.75,40.125)(44.25,24.25)
\qbezier(69,25.5)(89.5,40.125)(108,24.25)
\put(2.5,23.75){\makebox(0,0)[cc]{$P_{j+1}$}}
\put(66.25,23.75){\makebox(0,0)[cc]{$P_{j+1}$}}
\put(48.25,21.25){\makebox(0,0)[cc]{$Q_{j+1}$}}
\qbezier(1.5,49.5)(25.625,39)(49.25,48.5)
\put(47.25,44.25){\makebox(0,0)[cc]{$\Delta_j$}}
\put(-.75,51){\makebox(0,0)[cc]{$Y_j$}}
\put(52,49.5){\makebox(0,0)[cc]{$X_j$}}
\put(25.75,66){\vector(0,-1){50.5}}
\put(89.5,66){\vector(0,-1){50.5}}
\put(28,56.25){\makebox(0,0)[cc]{$g$}}
\put(91.75,56.25){\makebox(0,0)[cc]{$g$}}
\put(25.5,12.5){\makebox(0,0)[cc]{$A$}}
\put(89.25,12.5){\makebox(0,0)[cc]{$A$}}
\put(25.75,68.75){\makebox(0,0)[cc]{$B$}}
\put(89.5,68.75){\makebox(0,0)[cc]{$B$}}
\qbezier(77.5,18.75)(89.375,40.625)(110.75,54)
\put(113.75,55){\makebox(0,0)[cc]{$X_j$}}
\put(75.75,15.75){\makebox(0,0)[cc]{$Y_j$}}
\put(111.75,20.5){\makebox(0,0)[cc]{$Q_{j+1}$}}
\put(109,48.75){\makebox(0,0)[cc]{$\Delta_j$}}
\put(25.25,3.75){\makebox(0,0)[cc]{Fig. 3}}
\put(89.5,4){\makebox(0,0)[cc]{Fig. 4}}
\put(8,31.25){\makebox(0,0)[cc]{$\Delta_{j+1}'$}}
\put(70.75,31.25){\makebox(0,0)[cc]{$\Delta_{j+1}'$}}
\end{picture}

There are two cases to consider. 

Case 1. $\Omega_{j+1}$ is located at level $j+1$. There is a maximal element $\Delta_{j+1}''\in \mathscr{U}_{j+1}$ such that $\Delta_{j+1}''=\Delta_{j+1}'$. If Figure 3 occurs, then $\Delta_{j+1}''\cap 
\Delta_{j}\neq \emptyset$, $\partial \Delta_{j+1}''\cap \partial \Delta_j=\emptyset$ and $\Delta_{j+1}''\cup 
\Delta_{j}=\mathbf{H}$. From Lemma 4 of \cite{CZ1}, we deduce that $d_{\mathcal{C}}(u_{j+1}, u_j)\geq 2$. This is a contradiction. If Figure 4 occurs, then $\partial \Delta_j$ intersects $\partial \Delta_{j+1}''$, which implies that $\tilde{u}_{j+1}$ intersects $\tilde{u}_{j}$. Thus $u_{j+1}$ intersects $u_j$. This again contradicts that $d_{\mathcal{C}}(u_{j+1}, u_j)=1$. 

Case 2. $\Omega_{j+1}$ is located below  level $j+1$. This means that there is a maximal element $\Delta_{j+1}''\in \mathscr{U}_{j+1}$ that contains $\Delta_{j+1}'$. If Figure 3 occurs, then by the same argument as in Case 1, we deduce that $d_{\mathcal{C}}(u_{j+1}, u_j)\geq 2$. If Figure 4 occurs, then either $\partial \Delta_{j+1}''$ crosses $\partial \Delta_j$ or we have $\Delta_{j+1}''\cap 
\Delta_{j}\neq \emptyset$, $\partial \Delta_{j+1}''\cap \partial \Delta_j=\emptyset$ and $\Delta_{j+1}''\cup 
\Delta_{j}=\mathbf{H}$. In both cases, by the same argument as in Case 1, we deduce that  $d_{\mathcal{C}}(u_{j+1}, u_j)\geq 2$. This again contradicts that $d_{\mathcal{C}}(u_{j+1}, u_j)=1$. 

We conclude that all $\Omega_{k}$ with $k>j$ lie above level $k$. In particular, $\Omega_{m-1}$ is located above level $m-1$. So there is a maximal element $\Delta_{m-1}\in \mathscr{U}_{m-1}$ such that either $\partial \Delta_{m-1}$ lies above level $m$ or $\partial \Delta_m'$ crosses $\partial \Delta_{m-1}$. In both cases, by the same argument as in Case 1 and Case 2, we assert that $d_{\mathcal{C}}(u_{m-1}, u_m)\geq 2$. It follows that $s\geq m$. 
\end{proof}

Let $\tilde{u}_0\in \mathcal{C}_0(\tilde{S})$ and $\tilde{c}\in \mathscr{S}\backslash \mathscr{S}(2)$ be such that $i(\tilde{u}_0, \tilde{c})=1$. Let $u_0, g$, and $(\tau_0,\Omega_0,\mathscr{U}_0)$ be as before. Then $g$ possesses the property that $\varrho(\mbox{axis}(g))\cap \Omega_0\neq \emptyset$. As an easy consequence of Lemma \ref{L3.1}, we obtain
\begin{lem}\label{L3.2}
With the above conditions, if $\Omega_j$ is located above level $j$ for some $j$ with $1\leq j\leq s$, then the path $[u_0,\cdots, u_s,u_m]$, where $u_m=f^m(u_0)$ and $f=g^*$, is not a geodesic path. 
\end{lem}
\begin{proof}
By Lemma \ref{L3.1}, we assert that $s\geq m$. But from the assumption, we know that $i(\tilde{c},\tilde{u}_0)=1$, which means that $d_{\mathcal{C}}(u_0, f(u_0))=1$, and so for all $j$ with $0\leq j\leq m-1$, $d_{\mathcal{C}}(f^j(u_0), f^{j+1}(u_0))=1$. It follows from the triangle inequality that $d_{\mathcal{C}}(u_0,f^m(u_0))\leq m$. 
So by the definition, $[u_0,\cdots, u_s,u_m]$ is not a geodesic path. 
\end{proof}


\begin{lem}\label{L3.3}
With the same notations as in Lemma $\ref{L3.1}$, suppose that a path $[u_0,u_1,\cdots, u_s,u_m]$ is a geodesic path. Then $i(\tilde{c},\tilde{u}_0)=1$  if and only if all $\Omega_j$ are located at level $j$.
\end{lem}
\begin{proof}
By the same argument as in Lemma \ref{L3.2}, we obtain  
\begin{equation}\label{LP}
d_{\mathcal{C}}(u_0,f^m(u_0))\leq m. 
\end{equation}

If there is $\Omega_{j_0}$ that is located above level $j_0$, then by Lemma \ref{L3.1}, all $\Omega_j$ with $j\geq j_0$ are located above level $j$. By the same argument of Lemma \ref{L3.1}, we conclude that $d_{\mathcal{C}}(u_0,f^m(u_0))\geq m+1$. This contradicts (\ref{LP}). 

Conversely, if all $\Omega_j$ are located at level $j$, then for $j=0,\cdots, m-1$, $\Omega_j$ is adjacent to $\Omega_{j+1}$. By Lemma 2.1 of \cite{CZ11}, $d_{\mathcal{C}}(u_j,u_{j+1})=1$. Hence  $d_{\mathcal{C}}(u_0,f^m(u_0))=m$. By virtue of Lemma \ref{D}, we deduce that $i(\tilde{c},\tilde{u}_0)=1$. 
\end{proof}

\section{Proof of results}
\setcounter{equation}{0}

\noindent {\em Proof of Theorem $1.2$}:  Assume that $d_{\mathcal{C}}(u_0,f^m(u_0))=m$. If $i(\tilde{c},\tilde{u}_0)\geq 2$, then by Lemma \ref{D}, we have $d_{\mathcal{C}}(u_0,f^m(u_0))\geq m+1$. This is a contradiction. This shows that $i(\tilde{c},\tilde{u}_0)=1$. 

Conversely, suppose $i(\tilde{c},\tilde{u}_0)=1$. Let $[u_0,u_1,\cdots, u_s,u_m]$ be a geodesic segment joining $u_0$ and $u_m$. 	Then all $u_1,\cdots, u_s$ are non-preperipheral geodesic, which means that $\tilde{u}_1,\cdots, \tilde{u}_s$, are non-trivial, Let $(\tau_j,\Omega_j,\mathscr{U}_j)$ be the configurations corresponding to $u_j$.  By Lemma \ref{L3.3}, all $\Omega_j$ where $j=1,\ldots, s$, are located at level $j$. This implies that $\Omega_j$ is adjacent to $\Omega_{j+1}$ for $j=1,\cdots, s-1$. If $s\leq m-2$, then by the same argument of Lemma \ref{D}, $d_{\mathcal{C}}(u_s,u_m)\geq 2$. This is absurd. So $s\geq m-1$. 

On the other hand, if for all $j=1,2,\cdots, m-2$, $\Omega_j$ is adjacent to $\Omega_{j+1}$, then $\Omega_{m-1}$ is also adjacent to $\Omega_m$, 
which tells us that  $d_{\mathcal{C}}(\chi(\Omega_{m-1}),\chi(\Omega_m))=1$, that is, $d_{\mathcal{C}}(u_{m-1},u_m)=1$. It follows that $s=m-1$. In this case, 
$$
d_{\mathcal{C}}(u_0,f^m(u_0))=\sum_{j=0}^{m-1}d_{\mathcal{C}}(\chi(\Omega_j),\chi(\Omega_{j+1}))=m.
$$
Hence the geodesic segment connecting $u_0$ and $u_m$ is realized by the sequence $\Omega_0, \Omega_1,\cdots, \Omega_m$.  Note that $\chi(\Omega_j)=\chi(g^{j}(\Omega_0))=f^j(u_0)$. We conclude that the geodesic segment connecting $u_0$ and $u_m$ is
$$
[u_0,f(u_0), f^2(u_0),\cdots, f^{m-1}(u_0), f^m(u_0)].
$$

If there is another geodesic segment $[u_0, v_1,\cdots, v_{m-1}, u_m]$ connecting $u_0$ and $u_m$, then there is $j$, such that $v_j\neq f^j(u_0)$. Since $v_j$ for $j=1,\cdots, m-1$ are non-preperipheral, $\tilde{v}_j$ are all non-trivial geodesics, which allows us to define configurations $(\tau_j', \Omega_j', \mathscr{U}_j')$ corresponding to $v_j$. Then the assumption that $v_j\neq f^j(u_0)$ implies that $\Omega_j'$ is not located at level $j$. By the argument of Theorem 1.2 of \cite{CZ10}, $\Omega_j$ lies above level $j$. From the same argument of Lemma \ref{L3.1}, we conclude that $d_{\mathcal{C}}(u_0,f^m(u_0))\geq m+1$. This leads to a contradiction, proving that the geodesic segment connecting $u_0$ and $u_m$ is unique.     \qed
\medskip

\noindent {\em Proof of Theorem $1.1$}: Assume that $\tilde{c}\in \mathscr{S}\backslash \mathscr{S}(2)$. Choose $\tilde{u}_0\in \mathcal{C}_0(\tilde{S})$ so that $i(\tilde{c},\tilde{u}_0)=1$. Let $u_0\in F_{\tilde{u}_0}$ be such that $\Omega_0\cap \mbox{axis}(g)\neq \emptyset$, where $g\in G$ satisfies the condition $g^*=f$ and $(\tau_0,\Omega_0,\mathscr{U}_0)$ be the configuration corresponding to $u_0$. By Theorem 1.2, for every $m\geq 1$, $[u_0, f(u_0), \cdots, f^m(u_0)]$ and $[u_0, f^{-1}(u_0), \cdots, f^{-m}(u_0)]$ are the unique geodesic segments connecting $u_0, u_m$, and $u_0, u_{-m}$, respectively. 

We claim that $L_m=[f^{-m}(u_0), \cdots, f^{-1}(u_0), u_0, f(u_0), \cdots, f^m(u_0)]$ is a geodesic segment connecting $f^{-m}(u_0)$ and $f^{m}(u_0)$. Otherwise, the triangle inequality yields that $d_{\mathcal{C}}(f^{-m}(u_0),f^m(u_0))< 2m$. If  $L_m$ is not a geodesic segment, then since $f^m$ acts on $\mathcal{C}(S)$ as an isometry with respect to the path metric $d_{\mathcal{C}}$, $f^m(L_m)=[u_0,\cdots, f^{2m}(u_0)]$ would not be a geodesic segment, which contradicts Theorem 1.2. 

We conclude that $L_m$ is a geodesic path connecting $u_{-m}$ and $u_m$ for all $m>0$. To see that $L_m$ is the only geodesic segment joining $u_{-m}$ and $u_m$, we suppose there are 
two different geodesic segments $L_m$ and $L_m'$ joining $u_{-m}$ and $u_m$. Then since $f^m$ is an isometry, $f^m(L_m)$ and $f^m(L_m')$ would be two different geodesic segments connecting $u_0$ and $u_{2m}$, and this would contradict the uniqueness part of Theorem 1.2.  

It is now clear that both $f^{-m}(u_0)$ and $f^{m}(u_0)$ tend to the boundary $\partial \mathcal{C}(S)$ as $m\rightarrow +\infty$, and 
$$
\mathscr{L}_{u_0}=[\cdots, f^{-m}(u_0), \cdots, f^{-1}(u_0), u_0, f(u_0), \cdots, f^m(u_0), \cdots ]
$$ 
is an invariant bi-infinite geodesic under the action of $f^j$ for any $j$. We then define the map $\mathscr{I}$ by sending $\tilde{u}_0$ to $\mathscr{L}_{u_0}$. 

Let $u_0'\in F_{\tilde{u}_0}$ be such that $u_0\neq u_0'$ and $\mbox{axis}(g)\cap \Omega_0'\neq \emptyset$. We have $\tilde{u}_0=\tilde{u}_0'$. Hence $\Omega_0'\in \mathscr{R}_{\tilde{u}_0}$. By assumption we have $\mbox{axis}(g)\cap \Omega_0'\neq \emptyset$. Therefore, there is $j\in \mathbf{Z}$ such that $\Omega_0'=g^j(\Omega_0)$. This shows that $\mathscr{L}_{u_0}=\mathscr{L}_{u_0'}$. Thus the map $\mathscr{I}$ is well defined. 

Assume that  $\tilde{u}_0, \tilde{v}_0\in \mathcal{C}_0(\tilde{S})$ be such that $\tilde{u}_0\neq \tilde{v}_0$ and $i(\tilde{c},\tilde{u}_0)=i(\tilde{c},\tilde{v}_0)=1$. The vertices $u_0$ and $v_0\in \mathcal{C}_0(S)$ are so chosen that satisfy 

 (i) $u_0\in F_{\tilde{u}_0}$, $v_0\in F_{\tilde{v}_0}$, and \\
\indent (ii) $\Omega_{u_0}\cap \mbox{axis}(g)\neq \emptyset$ and $\Omega_{v_0}\cap \mbox{axis}(g)\neq \emptyset$.

\noindent By Theorem 1.1, we assert that
$$
\mathscr{L}_{v_0}=[\cdots, f^{-m}(v_0), \cdots, f^{-1}(v_0), v_0, f(v_0), \cdots, f^m(v_0), \cdots ]
$$ 
is also an invariant bi-infinite geodesic under the action of $f^j$ for any $j$. 

To show that $\mathscr{I}$ is injective, i.e., $\mathscr{L}_{u_0}\neq \mathscr{L}_{v_0}$, we only need to show that $v_0$ is not a vertex in $\mathscr{L}_{u_0}$. Suppose that $v_0=f^i(u_0)$ for some $m\in \mathbf{Z}$. Then since $f\in \mathscr{F}$, it is isotopic to the identity on $\tilde{S}$ as $x$ is filled in. It follows that $v_0$ is freely homotopic to $u_0$ if $u_0$ and $v_0$ are both viewed as curves on $\tilde{S}$. That is, $\tilde{u}_0=\tilde{v}_0$. This contradicts that $\tilde{u}_0\neq \tilde{v}_0$.   

The argument above also shows that $\mathscr{L}_{u_0}$ and $\mathscr{L}_{v_0}$ are disjoint bi-infinite geodesics in $\mathcal{C}(S)$. 

Since $\mathscr{F}^*$ is isomorphic to the fundamental group $\pi_1(\tilde{S}, x)$; it does not contain any elliptic elements. Thus (1) in Theorem 1.1 is a special case of Lemma 2.1.   \qed

\end{document}